\documentclass[a4paper]{amsart}
\usepackage[english]{babel}
\usepackage[utf8]{inputenc}
\usepackage[usenames,dvipsnames]{xcolor}
\usepackage{amsmath,amsfonts,amssymb,amsthm,mathtools, mathabx, amscd} 
\usepackage{graphicx}
\usepackage[colorinlistoftodos]{todonotes}    
\usepackage{setspace}
\usepackage{inconsolata} 
\usepackage{float}
\usepackage{hyperref}
\hypersetup{
     colorlinks = true,
     linkcolor = blue,
     anchorcolor = blue,
     citecolor = red,
     filecolor = blue,
     urlcolor = blue
     }
\usepackage{url}
\usepackage{cancel}
\usepackage{enumitem}
\usepackage{underscore}
\setlist[enumerate]{itemsep=1mm}
\usepackage{tikz, tikz-cd}
\usepackage{verbatim}
\usepackage[ruled,linesnumbered, inoutnumbered]{algorithm2e}
\usepackage[sort]{cite}
\providecommand{\keywords}[1]
{
  \small	
  \textbf{\textit{Keywords---}} #1
}
\usepackage{listings}
\definecolor{codedarkgreen}{RGB}{51, 133, 4}
\definecolor{codemaroon}{RGB}{133, 5, 63}
\definecolor{codeteal}{RGB}{0, 128, 96}
\lstdefinelanguage{Macaulay2}{
        classoffset=1,
        keywords={radical, ideal, map, kernel},
        keywordstyle={\color{blue}},
        classoffset=2,
        morekeywords={from, to, list},
        keywordstyle={\color{codemaroon}},
        classoffset=3,
        morekeywords={QQ},
        keywordstyle={\color{codedarkgreen}},
        classoffset=4,
        morekeywords={MonomialOrder},
        keywordstyle={\color{codeteal}}
}
\lstdefinelanguage{Julia}{
classoffset=1,
keywords={mixed_volume, using},
keywordstyle={\color{blue}}
}

\newcommand{\CC}{\mathbb{C}}
\newcommand{\NN}{\mathbb{N}}
\newcommand{\PP}{\mathbb{P}}
\newcommand{\ZZ}{\mathbb{Z}}

\theoremstyle{plain}
\newtheorem{thm}{Theorem}[section]
\newtheorem{lem}[thm]{Lemma}
\newtheorem{cor}[thm]{Corollary}
\newtheorem{prop}[thm]{Proposition}
\theoremstyle{definition}
\newtheorem{rem}[thm]{Remark}
\newtheorem{ex}[thm]{Example}
\newtheorem{definition}[thm]{Definition}

\title[Khovanskii bases for systems of polynomial equations]{Khovanskii bases for semimixed systems of polynomial equations -- a case of approximating stationary nonlinear Newtonian dynamics}

\author[V.~Borovik]{Viktoriia Borovik}
\address[Borovik]{University Osnabr\"uck,
Fachbereich Mathematik/Informatik
Albrechtstr.~28a,
49076 Osnabrück, Germany.}
\email{vborovik@uni-osnabrueck.de}

\author[P.~Breiding]{Paul Breiding}
\address[Breiding]{University Osnabr\"uck,
Fachbereich Mathematik/Informatik
Albrechtstr.~28a,
49076 Osnabrück, Germany.}
\email{pbreiding@uni-osnabrueck.de}
\thanks{V.B.\ and P.B.: Supported by Deutsche Forschungsgemeinschaft (DFG). Projektnr.\ 445466444.}

\author[J.~del~Pino]{Javier del Pino}
\address[del Pino]{Institute for Theoretical Physics, ETH Zürich, 8093 Zürich, Switzerland}
\email{jdelpino@phys.ethz.ch}
\thanks{}

\author[M.~Micha{\l}ek]{Mateusz Micha{\l}ek}
\address[Micha{\l}ek]{University of Konstanz, Dept.~of Mathematics and Statistics, Universit\"atsstr.~10, 78457 Konstanz, Germany.}
\thanks{M.M.: Supported by the DFG grant  467575307}
\email{mateusz.michalek@uni-konstanz.de}

\author[O.~Zilberberg]{Oded Zilberberg}
\address[Zilberberg]{University of Konstanz, Dept.~of Physics and Statistics, Universit\"atsstr.~10, 78464 Konstanz, Germany.}
\email{oded.zilberberg@uni-konstanz.de}
\thanks{}

\date{}

\begin{document}

\begin{abstract}
We provide an  approach to counting roots of polynomial systems, where each polynomial is a general linear combination of prescribed, fixed polynomials. Our tools rely on the theory of Khovanskii bases, combined with toric geometry, the Bernstein–Khovanskii–Kushnirenko (BKK) Theorem, and fiber products. 

As a direct application of this theory, we solve the problem of counting the number of approximate stationary states for coupled driven nonlinear resonators. 
We set up a system of polynomial equations that depends on three numbers $N,n$ and $M$ and whose solutions model the stationary states. The parameter $N$ is the number of coupled resonators, $2n-1$ is the degree of nonlinearity of the underlying differential equation, and $M$ is the number of frequencies used in the approximation. We use our main theorems, that is, the generalized BKK Theorem \ref{updatedBKK} and the Decoupling Theorem \ref{thm: Decoupling}, to count the number of (complex) solutions of the polynomial system for an arbitrary degree of nonlinearity $2n-1\geq 3$, any number of resonators $N\geq 1$, and $M=1$ harmonic. We also solve the case $N=1, n=2$ and $M =2$ and give a computational way to check the number of solutions for $N=1, n=2$ and $M >2$. This extends the results of~\cite{oscillators}. 
\end{abstract}

\maketitle

       \section{Introduction}

Counting solutions of a polynomial system is a common problem in mathematics with a broad variety of practical applications. However, equations in practical applications are usually highly structured, and general root counts such as Bézout's~\cite[Theorem 2.16]{jaBernd} or Bernstein–Khovanskii–Kushnirenko (BKK) Theorem~\cite{Bernshtein1975TheNO} tend to overestimate the actual number of solutions. The goal of this article is to develop techniques for counting complex solutions of such structured polynomial systems. 

One of our main findings is Theorem~\ref{updatedBKK}, which generalizes the BKK theorem. We summarize the statement of the theorem. Let
\begin{equation}\label{pol_system}
f_1^{(1)}(\mathbf x) = \cdots = f_{k_1}^{(1)}(\mathbf x)  = \cdots = f_1^{(r)}(\mathbf x)  =\cdots = f_{k_r}^{(r)}(\mathbf x) =0
\end{equation}
be a system of $m=k_1+\cdots+k_r$ polynomials in $m$ variables $\mathbf x = (x_1,\ldots,x_m)$, that are partitioned into $r$ blocks $f_1^{(i)},\ldots,f_{k_i}^{(i)}$ for $1\leq i\leq r$. Assume that there are~$r$ distinct finite sets of polynomials $\mathcal A_1,\ldots,\mathcal A_r$ such that the polynomials in the $i$-th block are linear combinations of those in $\mathcal A_i$. 
In Theorem~\ref{updatedBKK}, we establish conditions under which the number of complex solutions of the system~(\ref{pol_system}) can be expressed as a mixed volume of $m$ polytopes. To achieve this, we utilize the theory of Khovanskii bases~\cite{KM19}. A notable outcome of this theorem is Theorem~\ref{updatedKush}, which applies when there is only one block ($r=1$) and the polynomials in (\ref{pol_system}) are all linear combinations of the same set of polynomials.

We apply our results to a class of polynomial systems that we obtain from studying the stationary motion of driven coupled nonlinear resonators. The latter involves a set of time-dependent nonlinear ordinary differential equations that arise across multiple fields of natural science~\cite{Krack_2019,hb_julia_code}. The dynamics of such complex systems are approximated using a so-called Harmonic Balance ansatz~\cite{Krack_2019}. The condition for stationary motion entails a system of polynomial equations \cite{hb_julia_code} that depends on three parameters $N,n$ and $M$. The parameter $N$ is the number of coupled resonators, $2n-1$ is the degree of nonlinearity of the underlying differential equation, and $M$ is the number of frequencies used in the approximation. We use our main theorems (Theorems \ref{updatedBKK} and  \ref{updatedKush}) to count the number of (complex) solutions of the polynomial system for an arbitrary degree of nonlinearity $2n-1\geq 3$, any number of resonators $N\geq 1$, and $M=1$ frequency. We also solve the case $N=1, n=2$ and $M=2$. This extends the result in~\cite{oscillators}, where the authors computed the number of solutions in the case $n=2, M=1$ and $N$ arbitrary.

When we have at least two coupled resonators, the resulting polynomial system has a particular structure that we exploit further. In this setting, we can find another partition of the set~$\mathcal{A}:=\bigcup_{i=1}^r \mathcal{A}_i$ into $d$ sets $\mathcal{A}=\bigcup_{j=1}^d\mathcal{B}_j$ that do not share variables. We prove what we call the \emph{Decoupling Theorem} (Theorem \ref{thm: Decoupling}). More precisely, this theorem says that verifying a Khovanskii basis for (\ref{pol_system}) can be done by verifying a Khovanskii basis independently for each set $\mathcal{B}_j$, $j=1,...,d$. In our case, this significantly reduces the complexity of the problem, because we can ``decouple'' the involved computations and break them into simpler tasks.

Many of our motivations come from physics while our computational methods come from algebra and algebraic geometry. We therefore make an effort to include all basic definitions that are used throughout the article. Still, we advise the reader to consult~\cite[Chapters 1 and 2]{jaBernd} or \cite{cox1994ideals} for basics about ideals, rings and varieties. Readers not familiar with Gr\"obner bases or monomial algebras are referred to~\cite{sturmfels1996grobner} and \cite[Chapters 8 and 13.3]{jaBernd}. A very beautiful and comprehensive presentation of toric varieties, as well as basics of algebraic geometry, can be found in~\cite{CoxLittleSchenck}.

\subsection*{Outline}
The paper is structured as follows:  in Section~\ref{Section2}, we review the fundamental concepts of Khovanskii bases and methods for counting complex roots of polynomial systems, including the proof of Theorem~\ref{updatedBKK}. Section~\ref{Section3} is dedicated to exploring fiber products of unirational varieties and establishing the Decoupling Theorem (Theorem \ref{thm: Decoupling}). These two sections comprise the algebra part of the paper.

The remaining sections address the problem of approximating stationary nonlinear dynamics.
Section~\ref{Section4} involves the derivation of polynomial systems for the approximate stationary motion of coupled nonlinear resonators. In Section \ref{Section5}, we find a Khovanskii basis and the number of complex solutions for the nonlinear resonators problem with high nonlinearities ($N,n$ arbitrary and $M=1$) and apply Theorem~\ref{updatedKush}. Lastly, in Section \ref{Section6}, we apply Theorem~\ref{updatedBKK} to the theory of resonator networks with $N=1,n=2$ and $M=2$. 


\subsection*{Notation}
We write $\PP^a$ for the $a$-dimensional projective space, that is the space of lines in~$\CC^{a+1}$ through the origin. We interchangeably write $a_1\cdots a_k =\prod_{i=1}^k a_i$.
We use boldface to denote a sequence of elements, e.g.~$\mathbf{x}=(x_1,\dots,x_n)$. 
\bigskip
\section{Khovanskii bases and root counting}\label{Section2}

In this section, we recall the key definitions and properties of Khovanskii bases, as well as how we can use them for counting complex solutions of polynomial systems. For more details see~\cite{KKhov12} or \cite{KM19}.
 
Assume we have a polynomial system of $m$ equations 
\begin{equation}\label {simplesystem}
 f_1(\mathbf{x}) = \cdots = f_m(\mathbf{x}) = 0   
\end{equation}
in $m$ variables $\mathbf x = (x_1,\ldots,x_m)$, and we are interested in the number of its complex solutions. The BKK Theorem is the result of a series of articles~\cite{Kush1975, Bernshtein1975TheNO, Khovanskii1978NewtonPA}. It provides a way for computing an upper bound for the number of solutions in the torus $(\CC^{\times})^m$, where $\CC^{\times} = \CC\setminus \{0\}$. For the statement of the theorem, we recall that the \emph{Newton polytope} of a multivariate polynomial is $f=\sum_{\alpha\in A} c_\alpha \mathbf x^\alpha$, where $c_\alpha\neq 0$, is a convex hull of the exponents vectors $\alpha\in A\subset \mathbb Z^m$.  Given~$m$ polytopes~$P_1,\ldots,P_m\subset \mathbb R^m$ consider the Minkowski sum 
$$P(\boldsymbol \lambda):=\lambda_1 P_1 + \cdots + \lambda_m P_m,$$ where~$\boldsymbol \lambda =(\lambda_1,\ldots,\lambda_m)$. The scaled volume $m! \cdot  \mathrm{vol}(P(\boldsymbol \lambda))$ is a homogeneous polynomial of degree $m$ in~$\boldsymbol \lambda$, and the coefficient of the monomial $\lambda_1\cdots\lambda_m$ is called the \emph{mixed volume} of the~$P_i$'s,  denoted as $\mathrm{MV}(P_1, \ldots, P_m)$. For more information about the mixed volume of polytopes, we refer to~\cite{MixVol}.
\begin{thm}[The BKK Theorem]
\label{thm:BKK}
Let $f_1(\mathbf{x}) = \cdots =f_m(\mathbf{x}) = 0$ in $(\CC^{\times})^m$ be a system of polynomial equations. For every $1\leq i\leq m$ let $P_i$ denote the Newton polytope of $f_i$. 
\begin{enumerate}
\item 
The number of isolated zeros in~$(\CC^{\times})^m$ of this system, counting multiplicities, is bounded above by the mixed volume $\mathrm{MV}(P_1, \ldots, P_m)$. 
\item There is a proper algebraic subvariety in the space of coefficients of monomials appearing in the $f_i$'s, such that when the coefficients of the system are outside this subvariety, the number of isolated solutions in $(\CC^{\times})^m$ is exactly $\mathrm{MV}(P_1, \ldots, P_m)$ (we call such a system ``general'').
\item 
In particular, when $P:=P_1 =\ldots =P_m$, then the number of solutions of the general system is equal to $m!\cdot\mathrm{vol}(P)$.
\end{enumerate}
\end{thm}
\begin{proof}
See \cite[Theorem 1.2]{BKK}.
\end{proof}

\begin{rem}
The BKK Theorem holds for the more general case of \emph{Laurent polynomials}; i.e., polynomials with exponent vectors in $\mathbb Z^m$. 
\end{rem}

Theorem~\ref{thm:BKK} implies that the number of zeros of most polynomial systems is given by the mixed volume of their Newton polytopes. However, concrete examples are often not general and have the property that their number of solutions is much smaller than this mixed volume. Khovanskii bases can sometimes be used to solve this issue and provide a tighter upper bound. In the following, we first define the notion of Khovanskii basis and then prove a generalization of the BKK Theorem for Khovanskii bases (see Theorem~\ref{updatedBKK} below).

Let us denote the ring of polynomials in $\mathbf x$ with coefficients in the complex numbers $\mathbb C$ by $\CC[\mathbf x] = \CC[x_1,\ldots ,x_m]$  and consider $r$ distinct sets of polynomials 
$$\mathcal A_j = \{ h_{j,0},\ldots ,h_{j,a_j}\} \subset \CC[\mathbf x],\quad 1\leq j\leq r.$$
Assume that system~\eqref{simplesystem} consists of $r$ blocks with $k_j$ equations in each block. After relabeling we get (as in \eqref{pol_system}):
\begin{equation}\label{simplesystem_blocks}
f_1^{(j)}(\mathbf x) = \cdots = f_{k_j}^{(j)}(\mathbf x)  =0, \qquad j=1,\ldots, r.
\end{equation}
Furthermore, assume that each equation in the block $j$ is a linear combination of polynomials from $\mathcal A_j$:
\begin{equation}\label{linearcombination}
f_k^{(j)}(\mathbf x) = \sum\limits_{l=0}^{a_j} \lambda_{k,l}\, h_{j,l}(\mathbf x),\quad 1\leq k\leq k_j\,,  
\end{equation} 
where $\lambda_{k,l}\in\CC$ are coefficients.
As in \cite{BerndPolyhedralHomotopy}, we will call the system~\eqref{simplesystem_blocks} \textit{unmixed} when $r=1$, \textit{fully mixed} when $r=n$, and \textit{semimixed} for arbitrary $r$. We will now concentrate on semimixed systems, since fully mixed and unmixed systems are special cases of them. 

We write
$$\mathbf{P}^{a_1, \ldots , a_r} := \PP^{a_1}\times\dots\times \PP^{a_r}$$
and consider the rational map:
\begin{align*}
    \phi:\CC^m &\dashrightarrow \mathbf{P}^{a_1, \ldots , a_r}\\ 
    \mathbf{x}&\mapsto
    \big([h_{1,0}(\mathbf x) : \ldots : h_{1,a_1}(\mathbf x)],\ \dots\ ,\ [h_{r,0}(\mathbf x) : \ldots : h_{r,a_r}(\mathbf x)]\big).
\end{align*} 
The dashed arrow indicates that the map is defined outside the joint zero locus of~$\{h_{j,l}\}_{l=1}^{a_j}$ for each $j = 1, \ldots , r$. Next, we fix a monomial order~$\succ$ on $\CC[\mathbf x]$ and consider the corresponding map of leading terms relative to this order:
\begin{align*}
   &\phi_\mathrm{in}: \CC^m \dashrightarrow \mathbf{P}^{a_1, \ldots , a_r}\\ 
    &\mathbf{x}\mapsto
    \big([\mathrm{LT}(h_{1,0}(\mathbf x)):\ldots:\mathrm{LT}(h_{1,a_1}(\mathbf x))],\ \dots\ ,\ [\mathrm{LT}(h_{r,0}(\mathbf x)):\ldots:\mathrm{LT}(h_{r,a_r}(\mathbf x))]\big).
    \end{align*}
The Zariski closures, which are also the Euclidean closures in this case, of the images of these two maps are denoted by
$$X:= \overline{\operatorname{im} \phi} \quad \text{and}\quad Y:= \overline{\operatorname{im} \phi_\mathrm{in}}.$$

Without loss of generality, we can assume that homogeneous coordinate rings of varieties $X$ and $Y$ are embedded, as multigraded $\CC$-algebras, into a polynomial ring $\CC[\mathbf s, \mathbf x]$, where $\mathbf s = (s_1,\ldots,s_r)$. Denote
\begin{equation*}
    S:= \bigoplus\limits_{\alpha \in \NN^r}\mathbf s^{\alpha} V^{\alpha}
\end{equation*}
with $\mathbf s^{\alpha} = s_1^{\alpha_1}\cdots s_r^{\alpha_r}$, $V^{\alpha} = V_1^{\alpha_1}\cdots V_r^{\alpha_r}$, where $V_j = \operatorname{span}_{\CC}\{ \mathcal{A}_j\}$ and $V_j^{\alpha_j}$ is spanned by all $\alpha_j$-fold products of elements in $V_j$. We extend the monomial order on~$\CC[\mathbf x]$ to a degree compatible monomial order on $\CC[\mathbf s, \mathbf x]$, i.e. $\mathbf s^{\alpha}\mathbf x^{a} \succ \mathbf s^{\beta}\mathbf x^{b}$ whenever we have~$\alpha \succ \beta$ or $\alpha = \beta, a \succ b$. We slightly abuse the notation using $\succ$ as the monomial order on $\CC[\mathbf x]$ and $\CC[\mathbf s, \mathbf x]$ and $\CC[\mathbf s]$, using canonical inclusions, as well as on $\NN^m$, $\NN^{r+m}$ and $\NN^r$.
Then, the \textit{initial algebra} is spanned by the leading terms of all elements of $S$:
$$S_\mathrm{in} = \operatorname{span}_{\CC}\{\mathrm{LT}(f) \mid f \in S\}.$$
\begin{definition}[{\bf Khovanskii basis}]\label{def_Khovanskii_basis}
A subset of elements 
$G\subset S$ is called {\it a Khovanskii basis} if
the leading monomials of the elements in $G$ generate the initial algebra~$S_\mathrm{in}$.
\end{definition}
\begin{rem}
Khovanskii bases are defined for more general objects than subalgebras of polynomial rings through the concept of valuations. Nevertheless, in this paper, we work only with unirational varieties and deal only with subalgebras of polynomial rings, so in our case we always use a valuation coming from a monomial order on $\CC[\mathbf x]$.
\end{rem}

The following maps are called the \emph{Hilbert functions} of varieties $X$ and $Y$, respectively.
$$H_X\colon \ZZ_{\geq 0}^r \to \ZZ_{\geq 0}, \quad \mathbf{\alpha} = (\alpha_1, \ldots , \alpha_r) \mapsto \dim_{\CC}(\CC[X]_{\mathbf{\alpha}}),$$
$$H_Y\colon \ZZ_{\geq 0}^r \to \ZZ_{\geq 0}, \quad \mathbf{\alpha} = (\alpha_1, \ldots , \alpha_r) \mapsto \dim_{\CC}(\CC[Y]_{\mathbf{\alpha}}),$$
where the subscript $\mathbf{\alpha}$ denotes the part of the ring of degree $\mathbf{\alpha}$ and $\deg s^{\mathbf{\alpha}}\mathbf x^{a}=\alpha$, i.e.~the grading is only through the $\mathbf s$ variables. 

Just as in the univariate case, the functions $H_X(\alpha)$, $H_Y(\alpha)$ coincide with the integer-valued polynomials $HP_X(\alpha)$, $HP_Y(\alpha)$ for $\alpha \gg 0$, which can be written as follows (see \cite[Theorem 3.1]{CIDRUIZ2021136}):
\begin{align*}
HP_X (\alpha) &= \sum\limits_{i_1, \ldots , i_r \geq 0}\mathbf e(i_1, \ldots , i_r){\alpha_1 + i_1 \choose i_1} \cdots {\alpha_r + i_r \choose i_r},\\[0.2em]
HP_Y (\alpha) &= \sum\limits_{i_1, \ldots , i_r \geq 0}\mathbf e_\mathrm{in}(i_1, \ldots , i_r){\alpha_1 + i_1 \choose i_1} \cdots {\alpha_r + i_r \choose i_r},
\end{align*}

For all tuples $(i_1, \ldots , i_r)$ with $i_1 + \ldots  + i_r = \dim X$ the numbers $\mathbf e(i_1, \ldots , i_r)  \geq 0$ are called the \textit{mixed multiplicities} of $\CC[X]$ or \textit{multidegrees} of $X$. The geometric interpretation of these numbers is that they are equal to the number of points in the intersection of $X$ with the product 
$$L_1 \times \cdots \times L_r \subset \mathbf{P}^{a_1, \ldots , a_r},$$ where $L_j \subseteq \PP^{a_j}$ is a general linear subspace of codimension
$i_j$ for each $1 \leq j \leq~r$ (see \cite[Theorem 4.7]{CIDRUIZ2021136}). If the $L_j$'s are not general, the number of isolated points in the intersection $X\cap (L_1 \times \cdots \times L_r)$, is bounded by this multidegree.

Similarly, we call the numbers  $\mathbf e_\mathrm{in}(i_1, \ldots , i_r)$  multidegrees of $Y$, and we have a corresponding interpretation of the multidegrees as the number of points in the intersection of $Y$ with a product of linear subspaces.

The following theorem generalizes the BKK Theorem to Khovanskii bases.
\begin{thm}\label{updatedBKK}
Suppose that we have a semimixed system of polynomials as in Equation \eqref{simplesystem_blocks} with blocks of size $(k_1,\ldots,k_r)$ and that
    \begin{enumerate}
        \item the maps $\phi$ and $\phi_{\mathrm{in}}$ are generically injective;
        \item the polynomials $\bigcup_{j=1}^{r}\{s_jh \mid h\in \mathcal{A}_j\}\subset \CC[\mathbf s, \mathbf x]$ form a Khovanskii basis for the subalgebra~$S$.
    \end{enumerate}
    Then the number of isolated solutions of Eqs.~\eqref{simplesystem_blocks} is bounded by 
    $$\mathbf e(k_1, \ldots , k_r) = \mathbf e_\mathrm{in}(k_1, \ldots , k_r) = \mathrm{MV}(Q_1[k_1],\ldots, Q_r[k_r]),$$
    where $Q_j$ is the convex hull of the exponents of $\{\mathrm{LT}(h) \mid h\in \mathcal{A}_j\}$ for every $1 \leq j \leq r$ and
    \begin{align*}
     \mathrm{MV}(Q_1[k_1],\ldots, Q_r[k_r]):= \mathrm{MV}(\underbrace{Q_1,\ldots ,Q_1}_{k_1},\ldots ,\underbrace{Q_r,\ldots ,Q_r}_{k_r}).  
    \end{align*}
    Furthermore, for general $\lambda_{k,l}\in\CC$ from Eq.~\eqref{linearcombination}, the number of solutions of Eqs.~\eqref{simplesystem_blocks} is equal to $\mathrm{MV}(Q_1[k_1],\ldots, Q_r[k_r])$.
\end{thm}
\begin{proof}
For each $j=1,\ldots ,r$ there is a natural correspondence between the choice of $\lambda_{k,l}$'s and the choice of linear subspaces of codimension $k_j$: 
\[L_j=\left\{(p_1,\ldots ,p_{a_j})\in \mathbb{P}^{a_j}\ \Bigm|\ \sum_{l=1}^{a_j} \lambda_{1,l} p_l=\ldots =\sum_{l=1}^{a_j} \lambda_{k_j,l} p_l=0\right\}.\]
If we pick a solution of the system \eqref{simplesystem_blocks}, its image is a point in $X\cap (L_1\times \dots\times L_r)$. On the other hand, as $L_j$'s are general, a point in $X\cap (L_1\times \dots\times L_r)$ is a general point of $X$ and hence, by assumption $1$, has a unique inverse image under $\phi$. This inverse image must be a solution to the system \eqref{simplesystem_blocks}. Consequently, the required number of solutions equals to $\mathbf e(k_1, \ldots , k_r)$. 

Assumption 2 implies  $S_\mathrm{in} = \CC[Y]$.  As the initial algebra~$S_\mathrm{in}$ and the algebra~$S$ have the same Hilbert function (see \cite[Theorem 15.3]{Eisenbud}), we obtain 
$$\mathbf e(k_1, \ldots , k_r) = \mathbf e_\mathrm{in}(k_1, \ldots , k_r).$$ The last equality follows from the BKK Theorem (see also \cite[Theorem 4.14]{Khovanskii_intersection_index}).
\end{proof}
\begin{rem}
Assumption 1 can be relaxed to the assumption that the maps $\phi$ and $\phi_\mathrm{in}$ have finite degrees, where by the degree of a map we mean the cardinality of a general fiber. In this case, the number of solutions is equal to 
$$\deg (\phi) \cdot \mathbf e(k_1, \ldots , k_r)\,,$$
and the mixed volume must be divided by the degree of the map $\phi_{\mathrm{in}}$.

We also note that the degree of the latter map equals the index of the sublattice generated by the exponents of the monomials $\mathrm{LT}(h_{j,0}(\mathbf x))^{-1} \cdot \mathrm{LT}(h_{j,l}(\mathbf x))$, where $1\leq l\leq a_j$, inside the lattice $\ZZ^m$.
\end{rem}
As a direct corollary, we obtain the following theorem for unmixed systems, an alternative proof of which can be found, for example, in \cite{oscillators}.
\begin{thm}\label{updatedKush}
Consider $a+ 1$ polynomials $h_0, \ldots, h_a$ such that in \eqref{simplesystem} we have 
$$f_i(\mathbf{x}) = \sum_{l=0}^a \lambda_{i,l} h_l(\mathbf{x}),\quad 1\leq i\leq m;$$
i.e., we have $r=1$ block. Let $Q$ be a convex hull of the exponents of $\mathrm{LT}(h_l)$ for~$l=0,\ldots,a$.
Assume that in this setting
\begin{enumerate}
        \item the maps $\phi$ and $\phi_{\mathrm{in}}$ are generically injective;
        \item the polynomials $\{sh_0, \ldots , sh_a\}\subset \CC[s, \mathbf x]$ form a Khovanskii basis for the subalgebra they generate.
    \end{enumerate}
Then, the number of solutions of the system \eqref{simplesystem} for general $\lambda_{i,l}\in\CC$ equals 
$$m!\cdot \mathrm{vol}(Q) = \deg X = \deg Y,$$
where, as above, $X = \overline{\operatorname{im} \phi}$ and $Y = \overline{\operatorname{im} {\phi}_\mathrm{in}}$ are the image closures of the polynomial maps $\phi(\mathbf{x}) =\left[h_{0}:\cdots:h_{a}\right]$ and $\phi_\mathrm{in}(\mathbf{x})= \left[\mathrm{LT}(h_{0}):\cdots:\mathrm{LT}(h_{a})\right]$, respectively.
\end{thm}
To summarize, the two theorems above say that, using Khovanskii bases, we can replace each polynomial $h_i$ in the linear combination \eqref{linearcombination} by its leading term without changing the number of solutions. The new system will be a linear combination of monomials, for which we can apply the BKK Theorem. The following example shows that working with Khovanskii bases is necessary.
\begin{ex}
The polynomials $s(x^2+x), s(xy+y), s(y^2+1), s(y^3+2)\in\CC[s,x,y]$ do not form a Khovanskii basis. If we take two general linear combinations of $x^2+x$, $xy+y$, $y^2+1$, $y^3+2$, the system will have six solutions. However, if we take two general linear combinations of the leading monomials $x^2, xy, y^2, y^3$, we always obtain the solution $(0,0)$ and there will be two other solutions. We note that these are the leading monomials with respect to any term order.
\hfill $\diamond$ 
\end{ex}

\begin{rem}
 The mixed volume satisfies the so-called monotonicity property:
$\mathrm{MV}(P_1, \ldots , P_m) \leq \mathrm{MV}(Q_1,\ldots , Q_m)$, where $P_i \subseteq Q_i$. Hence, we always have 
$$\mathrm{MV}(P_1, \ldots , P_m) \leq m!\cdot \mathrm{vol}(\mathrm{conv}(P_1, \ldots , P_m)).$$
So the bound from Theorem \ref{updatedBKK} is tighter than in Theorem \ref{updatedKush}. Nevertheless, it can be more convenient in practice to assume that the system is unmixed by checking the simpler condition from Theorem \ref{updatedKush}. 
\end{rem}

The following example illustrates a case when the condition of Theorem \ref{updatedKush} is satisfied, but that of Theorem \ref{updatedBKK} is not.
\begin{ex}
Consider two polynomial equations
$$\lambda_1 + \lambda_2y + \lambda_3 (x^3 + xy^2) = \lambda_4 + \lambda_5x + \lambda_6 (x^2y + y^3)=0$$
in the two variables $(x,y)$ and with coefficients $\lambda_1,\ldots,\lambda_6\in\CC$.
The polynomials $s_1,\ s_1y,\ s_1(x^3 + xy^2),\ s_2,\ s_2x,\ s_2(x^2y + y^3)\in \CC[s_1,s_2,x,y]$ do not form a Khovanskii basis for the subalgebra they generate for any term order. On the other hand, the polynomials $s,\ sx,\ sy,\ s(x^3 + xy^2),\ s(x^2y + y^3)\in \CC[s,x,y]$ do form a Khovanskii basis for any term order.
\hfill $\diamond$ 
\end{ex}

At the end of this section, we recall a theorem of Kaveh and Manon that will help us to prove when a set of polynomials is  a Khovanskii basis. The result is based on the \emph{subduction algorithm}. For the statement, recall that a homogeneous ideal $I$ in a graded ring $R= \bigoplus_\alpha R_\alpha$ is an ideal generated by homogeneous elements; i.e., by elements that are contained in one of the $R_\alpha$.
\begin{thm}\label{thmKavehManon}
Let $Y$ be defined as above and $\{p_1, \ldots, p_k\}$ be generators of the homogeneous ideal $I(Y)$. Then 
$G:=\bigcup_{j=1}^{r}\{s_jh \mid h\in \mathcal{A}_j\}$ is a Khovanskii basis for the subalgebra~$S$, if and only if the remainder of $$p_l(s_1h_{1,0},\ \ldots ,\ s_1h_{1,a_1},\ \ldots ,\ s_rh_{r,0},\ \ldots ,\ s_rh_{r,a_r})$$ under subduction by $G$ computed in Algorithm~$\ref{subduction}$ is zero for all $l \in \{ 1, \ldots, k\}$.
\end{thm}
\begin{proof}
See \cite[Theorem 2.17]{KM19}.
\end{proof}
\begin{algorithm}[h!]\label{alg}
\caption{\label{subduction}The Subduction Algorithm}
\SetAlgoLined
\KwIn{A subset $ G = \{ g_1,\ldots , g_t\} \subset S$ and a polynomial $f \in S$.}
\KwOut{remainder of $f$ under subduction by $G$.}
\While{$f$ is not a constant in $\CC$}{
Find exponents $\alpha_1, \ldots, \alpha_{t} \in \NN_{0}$ and $c \in \CC^{\times}$, s.t.
$$\mathrm{LT}(f) = c\cdot \left(\mathrm{LT}(g_1)\right)^{\alpha_1}\cdots \left(\mathrm{LT}(g_t)\right)^{\alpha_{t}}\;$$
\eIf{no representation exists}{
\Return{$f$}\;
}{
Replace $f \leftarrow (f - g_1^{\alpha_1}\cdots g_t^{\alpha_t})$\;
} 
} 
\Return{$f$}
\end{algorithm}

\bigskip
\section{Fiber product of varieties}\label{Section3}
This section is motivated by the verification of condition 2 from Theorem~\ref{updatedKush} in the case when we have blocks of polynomials that do not share variables. In particular, this situation appears in the next section, where we study systems of polynomial equations arising from the method of \emph{Harmonic Balance} \cite{hb_julia_code}. 

For basics on polynomial rings, monomial orders, algebraic varieties or toric varieties, we refer to the textbooks \cite{CoxLittleSchenck, cox1994ideals, jaBernd, sturmfels1996grobner}.

Let $\mathbf{x}:=(x_1, \ldots, x_{k})$ and $\mathbf{y}:= (y_1, \ldots, y_{l})$ be two vectors of variables, such that~$x_i\neq y_j$ for all $i,j$. We consider two polynomial maps:
\begin{align*}
   \phi_1\colon \CC^{k+1} &\to \CC^{r+1}\\
   (s, x_1, \ldots, x_{k}) &\mapsto (s,\ s\cdot p_1(\mathbf x),\ \ldots,\ s\cdot p_r(\mathbf x)),
\end{align*}
where $p_i$ are polynomials in $\CC[\mathbf{x}]$ for all $i = 1, \ldots, r$ and
\begin{align*}
   \phi_2\colon \CC^{l+1} &\to \CC^{d+1}\\
   (s', y_1, \ldots, y_{l})&\mapsto (s',\ s'\cdot q_1(\mathbf y),\ \ldots,\ s'\cdot q_d(\mathbf y)),
\end{align*}
where $q_j$ are polynomials in $\CC[\mathbf{y}]$ for all $j = 1, \ldots, d$. We also define the joint parametrization map
\begin{align*}
   \phi\colon \CC^{k+l+1} &\to \CC^{r+d+1}\\
   (s, \mathbf x, \mathbf y) &\mapsto (s,\ s\cdot p_1(\mathbf x),\ \ldots,\ s\cdot p_r(\mathbf x),\ s\cdot q_1(\mathbf y),\ \ldots,\ s\cdot q_d(\mathbf y)).
\end{align*}

\smallskip 

Denote the Zariski closures of the images of these maps by
$$X := \overline{\operatorname{im}  \phi}, \qquad X_1 := \overline{\operatorname{im}  \phi_1}, \qquad X_2 := \overline{\operatorname{im}  \phi_2}.$$
In this section, our aim is to study the ideal 
$$I(X)\subseteq \CC[z,\mathbf{z},\mathbf{w}],$$
where $\CC[z,\mathbf{z},\mathbf{w}]:=\CC[z, z_1, \ldots, z_r, w_1, \ldots, w_d]$,
in terms of the ideals
\begin{align*}
&I_1:=I(X_1) \subseteq \CC[z_0,\mathbf{z}],\quad \text{and}\quad  I_2:=I(X_2) \subseteq \CC[w_0,\mathbf{w}].
\end{align*}

We note that $X$ is the variety corresponding to the reduced \emph{fiber product} \cite[Section 3]{Hartshorne} of $X_1$ and~$X_2$ over the affine line~$\CC$ with the maps given by projections~$\rho_1, \rho_2$ to the first coordinate. 
$$\begin{tikzcd}
& X \arrow{r}{\pi_2} \arrow{d}[swap]{\pi_1}
& X_2 \arrow{d}{\rho_2} \\
& X_1 \arrow[swap]{r}{\rho_1}
& \CC
\end{tikzcd}$$
Thus, we may naturally identify $I(X)$ with an ideal in $\CC[z_0,\mathbf{z},w_0,\mathbf{w}]/(z_0-w_0)$ and hence with an ideal in $\CC[z_0,\mathbf{z},w_0,\mathbf{w}]$ that contains $(z_0-w_0)$. Explicitly, as an ideal in $\CC[z_0,\mathbf{z},w_0,\mathbf{w}]$ the ideal $I(X)$ is the \emph{radical} of the sum of the extensions of the ideals $I_1=I(X_1)$, $I_2=I(X_2)$ by comorphisms $\pi_1^{\#}, \pi_2^{\#}$, respectively, and the ideal~$(z_0-w_0)$. 

Set-theoretically, $X$ is also the intersection of the product~$X_1\times X_2$ with the hyperplane $z_0-w_0=0$. 
We note that neither intersections with hyperplanes nor fiber products over affine lines are, in general, reduced. 
\begin{ex}
Let $\rho_1:\CC[z]\rightarrow \CC[z]/(z)$ and $\rho_2:\CC[z]\rightarrow \CC[x,z]/(z-x^2)$ be two morphisms defined by mapping $z$ to the class of $z$. The fiber product over $\CC[z]$ is:
\begin{align*}
\CC[z]/(z) \otimes_{\CC[z]} \CC[x,z]/(z-x^2)\ &\cong\ \CC[x,z]/(z,z-x^2)\\
&\cong\ \CC[x,z]/(z,x^2)\\
&\cong\ \CC[x]/(x^2),
\end{align*}
which is non-reduced. 

\hfill $\diamond$ 
\end{ex}

We will also give an example in our setting where the intersection of the affine variety with the hyperplane is non-reduced. To achieve this, in the next example we allow the polynomials~$\{p_i\}_{i=1}^r$ and $\{q_j\}_{j=1}^d$ to share the variables.
\begin{ex}
Consider the monomial map
\begin{align*}
   \phi\colon \CC^{4} &\to \CC^{6},\quad (s, x, s', y) \mapsto (s, sx^2, sy, s', s'x, s'y),
\end{align*}
and the corresponding affine toric variety that is the closure of the image of $\phi$. With \texttt{Macaulay2} \cite{M2} we compute its toric ideal:

\smallskip
\begin{lstlisting}[language=Macaulay2]
R = QQ[s, s', x, y]
S = QQ[z0..z2, w0..w2]
phi = map(R, S, {s, s*x^2, s*y, s', s'*x, s'*y})
kernel phi   
\end{lstlisting}
\smallskip

\noindent The result is $\langle z_2w_0 - z_0w_2,\ z_2w_1^2 - z_1w_0w_2,\ z_1w_0^2 - z_0w_1^2\rangle$.
Applying the command

\smallskip
\begin{lstlisting}[language=Macaulay2]
J = kernel phi + ideal {z0 - w0}
radical J == J
\end{lstlisting}
\smallskip

\noindent returns {\texttt{\textcolor{blue}{false}}}. This shows that the ideal is no longer radical after adding the generator $z_0-w_0$.

\hfill $\diamond$ 
\end{ex}

Let $^e$ denote the extension of an ideal by comorphism $\pi_i^{\#}$: $$I(X_i)^e\subseteq\CC[z_0, \mathbf{z}, w_0, \mathbf{w}], \quad i=1,2.$$
We have the following main result.
\begin{thm}\label{thm:generatorsTogether}
    We always have:
       \[I(X)=I_1^e+I_2^e+(z_0-w_0).\]
   Equivalently, in the ring $\CC[z, \mathbf{z}, \mathbf{w}]$, the ideal $I(X)$ is generated by the generators of~$I_1$ and the generators of $I_2$, where both variables $z_0$ and $w_0$ are mapped to~$z$. 
    \end{thm}
    \begin{proof}
Let us write $J:=I_1^e+I_2^e+(z_0-w_0)$.
Consider the map of $\CC$-algebras:
\[ \phi{^{\#}}:\CC[z_0, \mathbf{z}, w_0, \mathbf{w}]\rightarrow \CC[s,\mathbf{x},\mathbf{y}],\]
where $\phi{^{\#}}(z_0)=\phi{^{\#}}(w_0)=s$, $\phi{^{\#}}(z_i)=s\cdot p_i$ and $\phi{^{\#}}(w_j)=s\cdot q_j$. The statement is equivalent to showing that
$J=\ker \phi{^{\#}}$. It is clear that $J\subseteq \ker \phi{^{\#}}$ and that both ideals are homogeneous.

To prove the opposite inclusion, consider the following term order on the monomials in $\CC[z_0, \mathbf{z}, w_0, \mathbf{w}]$:
\begin{itemize}
    \item First, compare the degree of the monomials, i.e.~the order is degree compatible.
    \item Next, we consider \emph{degrevlex} order on $z_0,\mathbf{z}$, with $z_0$ the smallest variable. This monomial order first compares total degrees: 
    $$\qquad\text{If } \alpha_0+\cdots+\alpha_r > \beta_0+\cdots+\beta_r, \text{ then } z_0^{\alpha_0} z_1^{\alpha_1}\cdots z_r^{\alpha_r}>z_0^{\beta_0} z_1^{\beta_1}\cdots z_r^{\beta_r}.$$
     On the other hand, 
    $$\qquad\text{If }  \alpha_0+\cdots+\alpha_r = \beta_0+\cdots+\beta_r, \text{ then } z_0^{\alpha_0} z_1^{\alpha_1}\cdots z_r^{\alpha_r}>z_0^{\beta_0} z_1^{\beta_1}\cdots z_r^{\beta_r},$$ if and only if the smallest index $j$ where $\alpha_j-\beta_j\neq 0$ satisfies $\alpha_j-\beta_j<0$.
    \item Finally, when the monomials are the same in $z_0,\mathbf{z}$, we fix any term order on $w_0,\mathbf{w}$.
\end{itemize}

Suppose that $J\neq \ker \phi{^{\#}}$ and let $f\in \ker \phi{^{\#}}$ be a homogeneous polynomial with the smallest leading term, that does not belong to  $J$.
The contradiction will be obtained by reducing $f$ modulo $J$ to a polynomial with strictly smaller leading term.

First, we may assume $z_0$ does not appear in $f$, by using $z_0-w_0\in J$ (note that substituting $z_0$ by $w_0$ only decreases each monomial in the given term order).
Let us write $f=\sum_{i=1}^k m_i(\mathbf z)\cdot Q_i(w_0,\mathbf w)$, where $m_1,\ldots,m_k$ are monomials in $\mathbf{z}$ and $Q_1,\ldots,Q_k$ are homogeneous polynomials in $w_0,\mathbf{w}$. 

Without loss of generality, we may assume that $m_1$ is the largest among the~$m_i$'s. This implies that the leading term of $f$ is the product of $m_1$ and the leading term of~$Q_1$. If $Q_1\in I_2$, then $m_1Q_1\in I_2^e$ and we may reduce the leading term. Thus, we may assume $Q_1\not\in I_2$. This means that we can find scalars $\bar{s},\bar{\mathbf{y}}$ such that $Q_1(\bar{s},q_j(\bar{s}\bar{\mathbf{y}}))\neq 0$ and $\bar{s}\neq 0$. Consider a new, possibly non-homogeneous, polynomial:
\[\tilde f(\mathbf{z}):=f(\mathbf{z},\bar{s},q_j(\bar{s}\bar{\mathbf{y}}))=\sum_{i=1}^k m_i(\mathbf z)\cdot  Q_i(\bar{s},q_j(\bar{s}\bar{\mathbf{y}}))\in \CC[\mathbf{z}]. \]
It vanishes on $\{(\bar s \cdot  p_1(\mathbf{x}),\ldots, \bar s \cdot  p_r(\mathbf{x})) \mid \mathbf x\in\mathbb C^k\}\subset \CC^r$. Its leading monomial is $m_1$. 

Let $(\tilde f)_{z_0}$ be the homogenisation of $\tilde f$ with $\dfrac{z_0}{\bar s}$, which has the same leading monomial. The homogeneous polynomial $(\tilde f)_{z_0}$ vanishes on all points of the form $(s,s\cdot p_1(\mathbf x), \ldots, s\cdot p_r(\mathbf x))\in\mathbb C^{r+1}$ and thus $(\tilde f)_{z_0}\in I_1$. Finally, we use $(\tilde f)_{z_0}$ to decrease the leading term of $f$.
    \end{proof}

\begin{rem} In this remark, we present an alternative proof of Theorem \ref{thm:generatorsTogether} in the case when all $\{p_i\}_{i=1}^r$ and $\{q_j\}_{j=1}^d$ are monomials.

   Let $J\subset \CC[z, \mathbf{z},  \mathbf{w}]$ be the ideal generated by the generators of $I_1^e$ and $I_2^e$, after mapping $z_0$ and $w_0$ to $z$. Clearly, $J\subseteq I(X)$. We prove the opposite inclusion. 
As the variety $X$ is parameterized by monomials, $I(X)$ is generated by binomials; see \cite[Lemma~8.8]{jaBernd}. We note that all parameterizing monomials are of degree one in $s$, thus $I(X)$ is also homogeneous.  Let us consider a binomial generator: 
\[m_1-m_2\in \CC[z, \mathbf{z},  \mathbf{w}].\]
We have $m_1=a_1(z,\mathbf z)\cdot b_1(\mathbf w)$ and $m_2=a_2(z,\mathbf z) \cdot b_2(\mathbf w)$, where $a_1,a_2$ are monomials in the variables $z,\mathbf{z}$ and $b_1,b_2$ are monomials in the variables $\mathbf{w}$. If $\deg a_1=\deg a_2$, then $\deg b_1=\deg b_2$ and thus $a_1-a_2\in I_1$ and $b_1-b_2\in I_2$. This implies $m_1-m_2\in J$.  Without loss of generality, assume $d:=\deg a_1- \deg a_2>0$.
Then $a_1-a_2 z_0^d\in I_1$ and hence $f_1:=a_1b_1-a_2z^db_1\in J$. However, $z^db_1-b_2\in J$, as an extension of an element from~$I_2$. Thus, $f_2:=a_2z^db_1-a_2b_2\in J$. Finally, we have $m_1-m_2=f_1+f_2\in J.$ This shows $I(X)\subset J$.
\qed
\end{rem}
We note that similar results, even in the toric case, do not hold in general on the level of Gr\"obner bases. Example \ref{ex_similar_results_do_not_hold} below shows two ideals $J\subset \CC[z, \mathbf{z}]$ and $L\subset \CC[z,   \mathbf{w}]$, such that the union of Gr\"obner bases for $J$ and $L$ is not a Gr\"obner basis for $J+L\subset \CC[z, \mathbf{z},  \mathbf{w}]$. 
\begin{ex}\label{ex_similar_results_do_not_hold}
Consider the lex monomial order on $\CC[z, \mathbf{z},  \mathbf{w}]$, where $z$ is the largest variable and the variables $\mathbf{w}$ are smaller than $\mathbf{z}$. Then
\begin{align*}
&J = \langle z_1z_4 - z_3z_2,\ z^2z_3 - z_1^3,\ z^2z_4 - z_1^2z_2\rangle \subset \CC[z, \mathbf{z}],\\
&L = \langle w_1w_4 - w_3w_2,\ z^2w_3 - w_1^3,\ z^2w_4 - w_1^2w_2\rangle \subset \CC[z,   \mathbf{w}]
\end{align*}
form Gr\"obner bases with respect to this order for the ideals they generate. However, the element $z_2^3z_3^2w_4 - z_4^3w_1^2w_2 \in J + L$ has leading monomial $z_2^3z_3^2w_4$ which is not divisible by any leading monomial of the elements from the union of Gr\"obner bases.

\hfill $\diamond$ 
\end{ex}

However, we have the following positive result.
\begin{prop}
Let $G_i$ be a reduced degrevlex Gr\"obner basis for $I_i$ for $i=1,2$, where $z_0,w_0$ are the smallest variables. 
Then $G_1\cup G_2$ is a reduced degrevlex Gr\"obner basis for $I(X)$, where $z$ is the smallest variable and the variables $\mathbf{w}$ are larger than~$\mathbf{z}$.
\end{prop}
\begin{proof}
By Theorem \ref{thm:generatorsTogether},
 \[I(X)=I_1^e+I_2^e \subseteq \CC[z, \mathbf{z},  \mathbf{w}].\]
 We will check the Gr\"obner basis property by checking that all $S$-pairs reduce to zero. For any two elements in one $G_i$ the $S$-polynomial reduces to zero with respect to $G_i$ for each~$i=1, 2$. 
 
 Next, note that $\mathrm{LT}(f_1)$ and $\mathrm{LT}(f_2)$ cannot contain the variable $z$, since in this case all monomials of $f_{1}$ and $f_2$ contain it, which contradicts the fact that bases are reduced. Consider a pair $f_1 \in G_1$ and $f_2 \in G_2$. Since they are polynomials in different variables, their leading monomials are relatively prime, and by the second Buchberger's criterion the $S$-polynomial of this pair automatically reduces to zero.
\end{proof}
From Theorem \ref{thm:generatorsTogether} by induction, we obtain the following result.
\begin{thm}[{\bf Generators of the ideal of a fiber product}]\label{thm: ideal fiber product}
For $j=1,\dots, k$ let $X_j$ be the projective variety that is the closure of the image of 
    \[(s,\mathbf{x}_j)\mapsto (s,\ s\cdot p_{j,1}(\mathbf{x}_j),\ \ldots,\  s\cdot p_{j,i_j}(\mathbf{x}_j)).\]
    The ideal of the image of the closure of the map
    \[(s,\mathbf{x_1},\dots,\mathbf{x_k})\mapsto (s,\  sp_{1,1}(\mathbf{x}_1),\  \ldots,\ sp_{1,i_1}(\mathbf{x}_1),\ \ldots,\ sp_{k,1}(\mathbf{x}_k),\ \ldots,\ sp_{k,i_k}(\mathbf{x}_k))\]
    is generated by the union of generators of $I(X_j)$'s (identifying the first variable with the variable corresponding to $s$).
\end{thm}

\subsection{Khovanskii bases for fiber products of toric varieties}
Consider the system of polynomial equations
\begin{equation}\label{blocksys}
f_1(\mathbf x) = \cdots = f_m(\mathbf x) =0.
\end{equation}
We now use the results from the previous sections to give conditions under which a Khovanskii basis can be build from Khovanskii bases of subblocks of the system in (\ref{blocksys}). 
\begin{thm}[{\bf The Decoupling Theorem}]\label{thm: Decoupling}
Assume that the polynomials in the system (\ref{blocksys}) are linear combinations of polynomials in $\mathcal{A} \subset \CC[\mathbf x]$ and that $\mathcal{A}$ contains a constant. Suppose that $\mathcal A$ can be split into $d$ sets 
$$\mathcal{A} = \bigcup_{i=1}^d \mathcal{B}_i,$$
such that the polynomials from different sets $\mathcal{B}_i$ do not share variables with each other and each $\mathcal{B}_i$ contains a constant.
Suppose that for each $i = 1,\ldots,d$, polynomials from $G_i:=\{s_ih \mid h \in \mathcal{B}_i\}$ form a Khovanskii basis of the subalgebra they generate. Then, $$G:=\{sh \mid h \in \mathcal A\}$$  also form a Khovanskii basis for the subalgebra they generate. 
\end{thm}
\begin{rem}
The interpretation of the $d$ sets in this theorem is different than the interpretation of the $r$ blocks in Theorem \ref{updatedBKK}. First of all, in Theorem \ref{updatedBKK} the polynomials from different blocks $\mathcal A_i$ can share variables. Second, we will apply the Decoupling Theorem above only when considering unmixed (but not general) systems, so $r=1$ and $\mathcal{A}=\mathcal{A}_1$. 
\end{rem}
\begin{proof}[Proof of Theorem \ref{thm: Decoupling}]
    By Theorem \ref{thmKavehManon}, checking that the set of linearly independent polynomials is a Khovanskii basis consists of two parts:
\begin{itemize}
    \item Finding homogeneous binomial generators of the corresponding toric ideal; 
    \item Checking subduction to zero.
\end{itemize}
For each $i = 1, \ldots, r$ let 
$$\CC[s_i, s_ih_{i,1}, \ldots, s_ih_{i,j_i}] = \CC[z_{i,0}, \mathbf{z}_{i}]/I_i,$$ where $\mathbf{z}_{i} = (z_{i,1},\ldots ,z_{i,j_i})$ and 
$$\CC[s, sh_{1,1}, \ldots, sh_{1,j_1}, \ldots, sh_{r,1}, \ldots, sh_{r,j_r}] = \CC[z_{0}, \mathbf{z}_{1},\ldots, \mathbf{z}_{k}]/I.$$
By Theorem \ref{thm: ideal fiber product}, we know that the generating set of $I$ is a union of generators of~$I_1, \ldots, I_r$, assuming that $z_0 = z_{1,0} = \ldots = z_{r,0}$. Since for each ideal we know that all generators subduct to zero, this is also true for their union.
\end{proof}
Let us define for each set the polytope
\begin{equation}\label{def_Q_i}
Q_i = \mathrm{Conv}\{\alpha \mid \alpha \text{ is the exponent of } \mathrm{LT}(h), h \in \mathcal{B}_i\}.
\end{equation}
Using the assumption and the result of Theorem \ref{thm: Decoupling}, by Theorem~\ref{updatedKush}, the number of complex solutions of the polynomial system (\ref{blocksys}) is bounded by the normalized volume $m!\cdot\mathrm{vol}(Q)$, where~$Q$ is the \emph{subdirect sum} of the polytopes $Q_i$, see \cite[Definition 5.4]{oscillators}.
\begin{lem}
    $m!\cdot \mathrm{vol}(Q) = \prod\limits_{i=1}^{r} m!\cdot\mathrm{vol}(Q_i)$.
\end{lem}
\begin{proof}
See \cite[Lemma 5.5]{oscillators}.
\end{proof}
\begin{cor}\label{bound_for_same_blocks}
    When all the sets $\mathcal{B}_i$ in Theorem \ref{thm: Decoupling} have the same structure, the number of solutions of the system (\ref{blocksys}) is bounded by $(m!\cdot\mathrm{vol}(Q_1))^r$, where $Q_1$ is the polytope defined in (\ref{def_Q_i}).
\end{cor}
\bigskip
\section{Dynamical Systems and Nonlinear Differential Equations}\label{Section4}
We apply the theoretical framework from the previous sections to a counting problem in physics. We exemplify our technique on a parametrically-driven \emph{nonlinear resonator}. We show that our procedure provides tighter bounds than existing methods and can capture the system's complex behavior. 

Nonlinear differential equations are prevalent in physics, describing a wide range of physical phenomena, from classical mechanics to quantum field theory; see, e.g., \cite{Strogatz1994,Rand_2005,Dykman2012,Guckenheimer_2013}. They arise in many areas of physics, including fluid dynamics \cite{Falkovich2011}, electromagnetism \cite{Shen2002,Fallis2003}, and statistical mechanics. Nonlinear differential equations are essential for understanding complex phenomena such as turbulence, chaos \cite{Genesio1992}, and phase transitions \cite{griffin1996bose,Soriente21}. The study of nonlinear differential equations is an active area of research, with numerous challenges and opportunities for both analytical and numerical approaches \cite{Nayfeh_2008,allen2017computer,Krack_2019}. 

Our main focus is on differential equations of the form
\begin{equation}\label{diff_eq_example}
\ddot{X}+\omega_{0}^{2}\left(1-\lambda\cos\left(2t\omega\right)\right)X+ \gamma\dot{X}+F_{\text{nl}}(X)=0,
\end{equation}
which governs a nonlinear resonator amplitude $X=X(t)$ with natural frequency~$\omega_0$ subject to a parametric frequency modulation at frequency $2\omega$ with amplitude $\lambda$,  damping $\gamma$, and a nonlinear restoring force function $F_{\text{nl}}(X)$. We assume that the latter force depends only on odd powers of $X$, i.e., it is a $(2n-1)$-degree polynomial function
$$F_{\text{nl}}(X)=\alpha_1 X^3 + ... + \alpha_{n-1} X^{2n-1}.$$ The equation above appears in many areas of physics as a model for the behavior of systems subject to nonlinear restoring forces and damping. Examples include the dynamics of a mass-spring system subject to periodic forcing or the motion of a charged particle in a magnetic field. In addition, the equation has been used to describe the behavior of optical~\cite{Shen2002}, electrical~\cite{Heugel22}, and mechanical systems~\cite{Lifshitz_2008}, such as the dynamics of a laser cavity~\cite{Haken1975} or the vibrations of a nanomechanical resonator~\cite{Dykman2012}.

Analytical and numerical methods are commonly used to find stationary solutions of non-autonomous differential equations like \eqref{diff_eq_example}, where the system's behavior has the same or commensurate time dependence as the  external fields. Analytical methods include secular perturbation theory~\cite{Lifshitz_2008} and averaging Floquet expansions~\cite{richards2012analysis,Dykman2012}, which provide approximate or exact solutions for specific types of differential equations. Numerical methods, such as brute-force time-evolution~\cite{Chitra2015} relying on finite difference or spectral methods, are used when analytical solutions are not available. However, numerical methods can be computationally expensive, and their accuracy can be affected by choice of parameters and the numerical scheme. Therefore, the choice of the method depends on the specific problem and the desired accuracy. Here, we focus on an analytical method called \emph{Harmonic Balance}~\cite{Krack_2019,hb_julia_code}.

\subsection{Harmonic Balance}
The method of Harmonic Balance is a powerful technique used to find stationary solutions of nonlinear differential equations~\cite{Luo_2012,Krack_2019}. 
It has been used in various applications, including the analysis of nonlinear resonators, nonlinear waves, and nonlinear circuits~\cite{hb_julia_code}. The strength of this method lies in its ability to capture the essential features of the nonlinear dynamics while providing an explicit solution for the solution's frequency components.

The premise of Harmonic Balance is to represent the solution as a sum of harmonic functions with unknown coefficients. These coefficients can be found by equating the coefficients of the harmonics in the differential equation, resulting in a set of algebraic equations. We use Theorem \ref{updatedBKK}, Theorem \ref{updatedKush} and the theory of fiber products of toric varieties to count the number of (complex) solutions of these algebraic equations. 

The subset of real solutions correspond to the physical stationary solutions of the differential equation. Therefore, the number of real solutions of the polynomial equations resulting from the method of Harmonic Balance is crucial in understanding the system's behavior. The number of real solutions can determine whether the system exhibits intricate behavior, such as multiple steady states, limit cycles, or chaotic dynamics~\cite{Strogatz1994}. Understanding the number of real solutions can provide valuable insights into the system's dynamics and inform the design of control strategies for engineering applications.

\subsection{Polynomial Equations}

When seeking the stationary behavior ($t\rightarrow\infty$) of the equation \eqref{diff_eq_example}, the Harmonic Balance method (HBM) can be used to obtain approximate solutions that have a harmonic time dependence, e.g., at the frequency of time-dependent coefficients or fractions thereof. The HBM assumes that the solution can be represented by a finite Fourier series, and proposes a truncated Fourier ansatz of the form:
\begin{equation}\label{eq:ansatz_generic}
X(t) = \sum_{k=1}^{M} u_k(t) \cos(\omega_k t)+v_k(t) \sin(\omega_k t)=\sum_{k=1}^{M} (a_k(t)e^{i\omega_k t}+a^*_k(t)e^{-i\omega_k t}),
\end{equation}
where $M$ is the number of harmonics with frequencies $\omega_1,\dots, \omega_M$ included in the approximation, $u_k$ and $v_k$ are slowly-evolving real amplitudes, or ``quadratures'', that contain information about the amplitudes and phases of the harmonics. . Moreover, the complex amplitudes are defined by $a_k(t):=u_k(t)+iv_k(t)$. Their purpose is to facilitate algebraic manipulations. 

The stationary solutions of the system then follow from the substitution of the ansatz \eqref{eq:ansatz_generic} into equation \eqref{diff_eq_example} and isolating terms with a given frequency of interest, as described below. The replacement of the ansatz into the parametric drive term leads to wave-mixing contributions of the form
\begin{align}\label{eq:param_drive}
    \cos(2\omega t)X(t)=& \frac{1}{2}\sum_k a_{k}(t)\left(e^{i(\omega_{k}-2\omega)t}+e^{i(\omega_{k}+2\omega)t}\right)+\mathrm{c.c.}.
\end{align}
(Here, and also in the following, $\mathrm{c.c.}$ refers to the complex conjugate of the first term).

Therefore, the substitution reveals oscillatory terms at frequencies $\omega_k\pm 2\omega$. Similarly, substitution into the nonlinear force terms $X(t)^{2i-1}$ introduces the terms
\begin{align*}
\frac{1}{2^{2i-1}}\sum_{\substack{n_{1}+n_{2}+\cdots+n_{q}=n\\ n_{1},n_{2},\cdots,n_{q}\geq0}
}\left(\begin{array}{c}
2i-1\\
n_{1},n_{2},\cdots,n_{q}
\end{array}\right)\prod_{l=1}^{q}\sum_{j=0}^{n_{l}}\left(\begin{array}{c}
n_{l}\\
j
\end{array}\right)a_{l}^{n_{l}-j}\left(a_{l}^{*}\right)^{j}e^{it(n_{l}-2j)\omega_{l}},
\end{align*}
which contain sum and differences of $2i-1$ frequencies, accounting for all of the so-called $(2i)$-wave mixing processes~\cite{Shen2002}. For clarity, we can evaluate the expression above for $i=2$, corresponding to a cubic nonlinearity. This is usually the lowest-order correction to the linear, exactly-solvable limit~\cite{hb_julia_code}. The term $X(t)^3$ introduces the following terms:
\begin{equation}\label{eq:4_wave_mix}
    \begin{aligned}
        (a_q e^{i\omega_q t} + a_q^{*}e^{-i\omega_q t})^3 &;\\
        3(a_j e^{i\omega_j t} + a_j^{*}e^{-i\omega_j t})(a_l e^{i\omega_l t} + a_l^{*}e^{-i\omega_l t})^2 &;\\
        6(a_r e^{i\omega_r t} + a_r^{*}e^{-i\omega_r t})(a_j e^{i\omega_j t} + a_j^{*}e^{-i\omega_j t})(a_l e^{i\omega_l t} + a_l^{*}e^{-i\omega_l t}) &,
    \end{aligned}
\end{equation}
where $q, j, l, r \in\{1,...,M\}, j\neq l \neq r$. In other words, we have four-wave mixing contributions that oscillate at the frequency combinations $$\Delta\omega_{r,j,l}^{s_r,s_j,s_l}=s_r\omega_r+s_j\omega_j+s_l\omega_l,$$
where $s_r,s_j,s_{l}=\pm1$~\cite{Shen2002}.

The stationary solution of the equation \eqref{diff_eq_example} (also called stationary-motion condition) would directly follow from $\ddot{X}=\dot{X}=0$ in a time-independent (autonomous) differential equation. A way to simplify this calculation in the non-autonomous case is via Krylov-Bogoliubov time-averaging of the equation of motion \ref{diff_eq_example} with a time $T$~\cite{Rand_2005}. This approach approximates a system $\dot{X}=\varepsilon f(X,t,\varepsilon)$ dependent on a small parameter $\varepsilon\ll 1$ by the autonomous system~\cite{Rand_2005}
\begin{align}\label{eq:av_eq}
    \left(\begin{array}{c}
\dot{\mathbf{u}}\\
\dot{\mathbf{v}}
\end{array}\right)=\frac{\varepsilon}{T}\int_0^T f(\mathbf{u},\mathbf{v},t,0) dt\equiv \varepsilon\bar{f}(\mathbf{u},\mathbf{v}),
\end{align}
where the $2M$-vectors $\mathbf{u}=(u_1,\dots,u_M)$ and $\mathbf{v}=(v_1,\dots,u_M)$ regroup the harmonic ansatz amplitudes.  The averaging method relies on the cancellation of weak oscillations over one period, leaving only the constant averaged terms that contribute to long-term behavior. This method is valuable when the nonlinear system has a separation of time scales, with fast oscillations on a much shorter scale than the slow dynamics of $\mathbf{u}$ and $\mathbf{v}$. Importantly, the dynamical equation for $(\mathbf{u},\mathbf{v})$ captures the asymptotic behavior of the original system. Additionally, the solutions $\mathbf{u}$ and~$\mathbf{v}$ provide approximate values of $X$ over a finite time period inversely proportional to the slow time scale, $1/\varepsilon$. Higher orders can be explored using the so-called \emph{near-identity transformation}~\cite{Holmes1981}.

Without further specification of the ansatz frequencies $\omega_k$'s or the averaging time $T$, no harmonics can be allowed to have finite amplitudes in the stationary limit, namely  $a_k=0$ ($u_k = v_k=0$) is the only possible solution. Non-trivial solutions can be revealed by the remaining terms in the equation \eqref{eq:av_eq} after selecting an appropriate value for $T$. The simplest nontrivial example arises when solving the homogeneous differential equation with $\gamma=0$. The stationary solution, in this case contains only the single harmonic at frequency $\omega_1=\omega_0$ i.e. $X_0(t)=a_1e^{i\omega_0 t}+\mathrm{c.c.}$, and $a_{k>1}=0$. The solution $X_0(t)$ (natural oscillations) can be found from substitution of the ansatz \eqref{eq:ansatz_generic} and choosing  the natural period $T=2\pi/\omega_0$ as the averaging period. The averaging method can be used to introduce corrections to such natural oscillations.  Assuming the system is initially therefore described by the single harmonic with $\omega_1=\omega_0$, the impact of parametric drive will be relevant when $\omega_1=\omega_0\approx 2\omega$: the resulting term in the equation \eqref{eq:param_drive} will evolve slowly with time and prevail after averaging (indeed, for an arbitrary long averaging time~$T\rightarrow\infty$ if the condition is exact $\omega_0= 2\omega$). The internal frequency relationship $\omega_0= 2\omega$ exemplifies a resonant condition known as \emph{parametric resonance}.
Parametric resonance amplifies and attenuates oscillatory motion at frequency $\omega_0$. More intricate frequency conditions, often referred to as "multiphoton" resonances in nonlinear optics~\cite{Shen2002}, can enhance nonlinear wave-mixing processes in the equation \eqref{eq:4_wave_mix}. Generally, a nonlinear contribution becomes significant when averaged over a time $T_{r,j,l}^{s_r,s_j,s_l}=2\pi/\Delta\omega_{r,j,l}^{s_r,s_j,s_l}\gg 1$ provided that $\omega_{r,j,l}^{s_r,s_j,s_l}\simeq 0$. Furthermore, we reduce to terms oscillating at frequencies $\Delta\omega_{r,j,l}^{+,-,\pm}$ with $\omega_r=\omega_j$, which describe self-phase and cross-phase modulation processes in nonlinear optics. Here, the phase of propagating waves within a material is influenced by the material's intensity-dependent refractive index \cite{kippenberg2004nonlinear}. 

The time-averaging process now converts the stationary-motion condition into a set of polynomial equations. We aim to find all solutions for these equations. In this article, we count the number of complex solutions of the corresponding system for two cases: 
\begin{itemize}
    \item $M=1$ and arbitrary $n, N$;
    \item $n=2, N = 1$ and arbitrary $M$.
\end{itemize}
The equations of motion obtained from this averaging procedure are equivalent to those found via the Harmonic Balance step.

\subsubsection{One resonator and $M=1$ frequency}
In the simplest version of the ansatz, we include a single frequency ($M=1$).
For illustration purposes, we first consider a cubic nonlinearity ($2n-1=3$), which is often the lowest-order correction to the linear equation of motion of a harmonic oscillator. It captures the leading nonlinear effects that become significant when the amplitude of oscillation becomes larger. Using the ansatz, we obtain
\begin{align*}
&\left[\ddot{u}+\gamma\dot{u}+2\omega\dot{v}+u\left(\omega_{0}^{2}-\omega^{2}\right)+\gamma\omega v-\frac{\lambda\omega_{0}^{2}}{2}u+\frac{3\alpha\left(u^{3}+uv^{2}\right)}{4}\right]\cos(\omega t)\nonumber\\
+&\left[\ddot{v}+\gamma\dot{v}-2\omega\dot{u}+v\left(\omega_{0}^{2}-\omega^{2}\right)-\gamma\omega u+\frac{\lambda\omega_{0}^{2}}{2}v+\frac{3\alpha\left(v^{3}+u^{2}v\right)}{4}\right]\sin(\omega t)\nonumber\\
+&\left[\frac{\alpha\left(u^{3}-3uv^{2}\right)}{4}-\frac{\lambda\omega_{0}^{2}}{2}u\right]\cos(3\omega t)+\left[\frac{\alpha\left(3u^{2}v-v^{3}\right)}{4}-\frac{\lambda\omega_{0}^{2}}{2}v\right]\sin(3\omega t)=0.
\end{align*}
To identify the stationary condition of this system, we impose Harmonic Balance by matching terms with the same oscillatory dependence on both sides of the equation. This matching requires the same hypothesis as the averaging method: the slow evolution of the amplitudes $u(t)$ and $v(t)$ in relation to the ansatz periods $2\pi/\omega_k$. In line with this assumption, note that the terms that oscillate at $3\omega$ the $X^{3}$ term generates  (frequency up-conversion) are disregarded. Such  high-frequency contributions they are not properly balanced under the current ansatz. The resulting conditions for $u,v$ read $a_0 + a_1 u + a_2 v + a_3 u(u^2 + v^2) =b_0 + b_1 u + b_2 v + b_3 v(u^2 + v^2) =0$
with coefficients $a_0,\ldots,b_3$ that depend on the parameters $\alpha, \omega,\omega_0,\lambda$.

The generalization to higher nonlinearities, i.e., $2n-1$ is the highest degree of~$X(t)$ in the equation \eqref{diff_eq_example}, produces  polynomial equations of the form 
\begin{equation} \label{equations_one_oscillator}
\begin{aligned}
a_0 + a_1 u + a_2 v + a_3 u(u^2 + v^2) + \ldots + a_{n+1}u(u^2 + v^2)^{n-1} &= 0, \\
b_0 + b_1 u + b_2 v + b_3 v(u^2 + v^2)  + \ldots + b_{n+1}v(u^2 + v^2)^{n-1} &=0.
\end{aligned}
\end{equation}
In order to prove this statement, it is sufficient to focus on the expression that results from the insertion of the ansatz for $X(t)$ in a generic term $X^{2n-1}$.
We express $X(t)$ in complex form
\begin{equation*}
    X(t)=u\cos(\omega t)+v\sin(\omega t)=\frac{1}{2}(ae^{i\omega t}+a^{*}e^{-i\omega t}),
\end{equation*}
where $a=u+iv$.  Note that the extraction of the 
respective coefficients of $\cos(\omega t)$ and $\sin(\omega t)$ in the Harmonic Balance approach is equivalent to filtering the equations with a window function $e^{i\omega t}$ along the timescale $2\pi/\omega$, i.e., a finite time Fourier transform.

\begin{lem}\label{Fourier1}
   The Fourier transform $(ae^{i\omega t}+a^{*}e^{-i\omega t})^{2k+1}$ over the  oscillation period $T=2\pi/\omega$ yields ${2k+1 \choose k} a|a|^{2k}$.
\end{lem}
\begin{proof}
    Applying the binomial theorem, we obtain
\begin{equation*}
\begin{aligned}
    (ae^{i\omega t}+a^{*}e^{-i\omega t})^{2k+1} =\sum_{j=0}^{2k+1}\left(\begin{array}{c}
2k+1\\
j
\end{array}\right)a^{2k+1-2j}|a|^{2j}e^{i(2k+1-2j)\omega t}.
\end{aligned}
\end{equation*}
We can succinctly express the Fourier transform as
\begin{equation*}
\begin{aligned}
    I_k(\omega)=\frac{\omega}{2\pi}\int_0^{2\pi/\omega} \text{d}t\hspace{1mm}(ae^{i\omega t}+a^{*}e^{-i\omega t})^{2k+1} e^{-i\omega t}\approx &\sum_{j=0}^{2k+1}\left(\begin{array}{c}
2k+1\\
j
\end{array}\right)a^{2k+1-2j}|a|^{2j}\delta_{j,k}\\
=&{2k+1 \choose k} a|a|^{2k}.
\end{aligned}
\end{equation*}
where we defined the Kronecker delta symbol $\delta_{j,k}=\frac{1}{2\pi}\int_0^{2\pi}\text{d}\varphi\hspace{1mm}e^{i(j-k)\varphi}$.
\end{proof}
The resulting Harmonic Balance equations are 
\begin{equation}\label{FG_original}
\begin{aligned}
    \left(\frac{\omega_{0}^{2}-\omega^{2}}{2}-\frac{\lambda\omega_{0}^{2}}{4}\right)u+\frac{\gamma\omega}{2} v+\sum_{k=1}^{n-1}\alpha_k F_k(u,v)&= 0,\\	
\left(\frac{\omega_{0}^{2}-\omega^{2}}{2}+\frac{\lambda\omega_{0}^{2}}{4}\right)v-\frac{\gamma\omega}{2} u+\sum_{k=1}^{n-1}\alpha_k G_k(u,v)&= 0,
\end{aligned}
\end{equation}
where the terms in the sum are given by $F_k(u,v)=\mathrm{Re}(I_k(\omega))$ and $G_k(u,v)=\mathrm{Im}(I_k(\omega))$ with $I_k(\omega)=\frac{\omega}{2\pi}\left[\int_{0}^{2\pi/\omega}\text{d}t\hspace{1mm}X(t)^{2k+1}e^{-i\omega t}\right]$.
By Lemma \ref{Fourier1},
\begin{align*}    F_k(u,v)=&\frac{1}{2^{2k+1}}\left(\begin{array}{c}
2k+1\\
k
\end{array}\right)u(u^2 + v^2)^{k},\\
G_k(u,v)=&\frac{1}{2^{2k+1}}\left(\begin{array}{c}
2k+1\\
k
\end{array}\right)v(u^2 + v^2)^{k}.
\end{align*}
Thus, we obtained algebraic equations exactly of the form \eqref{equations_one_oscillator}, where the coefficients $a_0,...,b_{n+1}$ satisfy the linear relations
\begin{equation}\label{linear_relations}
a_1 +b_2 = \omega_0^2 - \omega^2,\quad a_2+b_1 = 0\quad \text{and}\quad a_k = b_k \text{ for $k=3,...,n+1$}.
\end{equation}

\subsubsection{$N$ coupled resonators and $M=1$ frequency}

Coupled resonators are described by the following system of differential equations, extending (\ref{diff_eq_example}):
\begin{equation*}
\ddot{X_i}+\omega_{0}^{2}\left(1-\lambda\cos\left(2t\omega\right)\right)X_i+ \gamma\dot{X_i}+\alpha_1 X_i^3 + ... + \alpha_{n-1} X_i^{2n-1} + \sum_{l\neq i}J_{i,l}X_l=0,
\end{equation*}
 where $J_{i,l}$ is a new matrix of parameters representing the coupling between our~$N$ resonators. For $N$ coupled resonators with $(2n-1)$-degree nonlinearity and $1$~leading frequency this leads to  $2N$ algebraic equations in $2N$ variables $u = (u_1,...,u_N)$, $v= (v_1,...,v_N)$, which come in $N$ pairs: 
\begin{equation}\label{equations_N}
\begin{aligned}
& \begin{bmatrix}
p_1(u,v), &q_1(u,v),&\hdots ,&p_N(u,v), &q_N(u,v)    
\end{bmatrix}=0, \text{ where } \\
p_{i}(u,v) &= a_{0,i} + a_{1,i} u_i + a_{2,i} v_i + \ldots + a_{n+1,i}u_i(u_i^2 + v_i^2)^{n-1} + \frac{1}{2}\sum_{j\neq i}J_{j,i}u_j,\\
q_{i}(u,v)&= b_{0,i} + b_{1,i} u_i + b_{2,i} v_i + \ldots + b_{n+1,i}v_i(u_i^2 + v_i^2)^{n-1} +  \frac{1}{2}\sum_{j\neq i}J_{j,i}v_j,\\
\end{aligned}
\end{equation}
where $a_{1, i} + b_{2, i} = \omega_0^2 - \omega^2, \; a_{2,i} = -b_{1,i}\; \text{and}\; a_{k,i} = b_{k,i}$ for $i=1,\dots,N$ and $k = 3,\dots,n+1$.


\subsubsection{One resonator with $n=2$ and $M>1$ frequencies}
For clarity, we will consider only differential equations with a cubic nonlinearity, i.e.,
\begin{equation*}\label{diff_eq_cubic}
\ddot{X}+\omega_{0}^{2}\left(1-\lambda\cos\left(2t\omega\right)\right)X+ \gamma\dot{X}+\alpha X^3 =0.
\end{equation*}

Now, we return to the general ansatz of the form \eqref{eq:ansatz_generic}, where $\omega_1=\omega$ and $\omega_k/\omega$ is irrational for $k> 1$ (i.e., the frequencies considered are incommensurate with  $\omega$ and with each other). Hence, we assume the frequencies $\omega_k$ are isolated from one another, i.e., they do not form a continuous band.

The system of polynomial equations that arises from the coefficients of the harmonics with frequencies $\omega_k$'s, is now defined in $2M$ variables $u = (u_1,\dots,u_M)$, $v = (v_1,\dots,v_M)$ that appear in $M$ pairs:
$$\begin{bmatrix}
p_1(u,v), &q_1(u,v),&\dots ,&p_M(u,v), &q_M(u,v)    
\end{bmatrix}=0.$$
Each polynomial pair can be written as
\begin{align*}
p_k(u,v)&=\left(\frac{\omega_{0}^{2}-\omega^{2}}{2}-\frac{\lambda\omega_{0}^{2}}{4}\delta_{k,1}\right)u_k +\frac{\gamma\omega}{2} v_k+\alpha F_k(u,v),\\	q_k(u,v)&=\left(\frac{\omega_{0}^{2}-\omega^{2}}{2}+\frac{\lambda\omega_{0}^{2}}{4}\delta_{k,1}\right)v_k -\frac{\gamma\omega}{2} u_k+\alpha G_k(u,v).
\end{align*}
The linear terms in $u_m$ and $v_m$ within the parentheses, related to the parametric-frequency modulation equation \eqref{eq:param_drive}, are similar to those in the $M=1$ case. When averaging over such a term, we assumed a long averaging timescale $T\gg1$ such that only the contribution from the first harmonic $\omega_1=\omega$ prevails.

The additional terms in $p_k(u,v)$ and $q_k(u,v)$, namely $F_k(u,v)$ and $G_k(u,v)$  are degree three polynomials in analogy to the result of Lemma~\ref{Fourier1} applied to \eqref{eq:4_wave_mix}.  To comply with our framework, we can simplify matters and retain terms oscillating with frequencies $|\omega_r-\omega_j\pm\omega_l|$, and further assume the \textit{degenerate} case where at least two frequencies are nearly equal ($\omega_r\approx\omega_j$). As previously mentioned, these terms describe self-phase and cross-phase modulation nonlinear processes \cite{kippenberg2004nonlinear}. 

Under these assumptions, averaging over the long timescale compared with other periods $T^{+,+,\pm}_{r,j,l}=2\pi/|\omega_r+\omega_j\pm\omega_l|$ keeps only contributions from the cube 
$$(a_k e^{i\omega_k t} + a_k^{*}e^{-i\omega_k t})^3$$ 
from \eqref{eq:4_wave_mix} and summands such as 
\begin{equation}\label{terms}
   3(a_k e^{i\omega_k t} + a_k^{*}e^{-i\omega_k t}) \cdot 2  |a_l|^2, \quad k\neq l\in\{1,...,M\}.
\end{equation} 
 By Lemma \ref{Fourier1}, the relevant averages give us $3 a_k|a_k|^2$ and $3\cdot 2 \cdot a_k|a_l|^2$ of the terms~\eqref{terms}. Therefore,
\begin{align*}
  F_k(u,v) &\approx \frac{3}{8}u_k(u_k^2 + v_k^2) + \frac{3}{4}\sum_{l\neq k}u_k(u_l^2 + v_l^2),\\ 
  G_k(u,v) &\approx \frac{3}{8}v_k(u_k^2 + v_k^2) + \frac{3}{4}\sum_{l\neq k}v_k(u_l^2 + v_l^2).
\end{align*}
As a result, the corresponding system of algebraic equations has the form:
\begin{equation}\label{equations_M}
    \begin{aligned}
        & \begin{bmatrix}
p_1(u,v), &q_1(u,v),&\hdots ,&p_M(u,v), &q_M(u,v)    
\end{bmatrix}=0, \text{ where}\\
p_k(u,v)&= a_{k,0} +  a_{k,1} u_k + c_k v_k + d_k u_k(u_k^2+v_k^2) + 2d_k \sum\limits_{j\neq k} u_k(u_j^2 + v_j^2),\\	
q_k(u,v)&= b_{k,0} - c_k u_k +  b_{k,1} v_k + d_k v_k(u_k^2+v_k^2) + 2d_k \sum\limits_{j\neq k} v_k(u_j^2 + v_j^2).\\		
    \end{aligned}
\end{equation}

\bigskip
\section{Algebraic Degree of Coupled resonators: Higher Nonlinearities}\label{Section5}

We now count the number of complex solutions of the system of polynomial equations \eqref{equations_one_oscillator} that arise for $N=1$ resonator, $M=1$ frequency and nonlinearity of degree $2n-1\geq 3$.
\begin{thm}\label{thm:count1}
For general $a_0,\ldots, b_{n+1}$ the number of complex solutions $(u,v)\in \mathbb C^2$ of the system of polynomial equations
\begin{equation*}
\begin{aligned}
a_0 + a_1 u + a_2 v + a_3 u(u^2 + v^2) + \ldots + a_{n+1}u(u^2 + v^2)^{n-1} &=0\\
b_0 + b_1 u + b_2 v + b_3 v(u^2 + v^2)  + \ldots + b_{n+1}v(u^2 + v^2)^{n-1} &=0\\
\end{aligned}
\end{equation*}
is $4n-3$. This also holds if $a_0,\ldots, b_{n+1}$ are determined by a general choice of the physical parameters $\omega_0,\omega,\lambda,\gamma,\alpha_1,\ldots\alpha_k$ from (\ref{FG_original}).
\end{thm}

We prove Theorem \ref{thm:count1} in the next two subsections. First, however, let us use the theorem together with the results from Section \ref{Section3} to compute the general number of solutions of the system \eqref{equations_N}, where $N$ systems of the form in Theorem \ref{thm:count1} are coupled. In this case, we have in total $2N$ equations in $2N$ variables $u = (u_1,...,u_N)$ and~$v= (v_1,...,v_N)$, which come in $N$ pairs:
\begin{equation}\label{N_oscillators}
\begin{aligned}
& \begin{bmatrix}
p_1(u,v), &q_1(u,v),&\hdots ,&p_N(u,v), &q_N(u,v)    
\end{bmatrix}=0, \text{ where }\\
p_{i}(u,v) &= a_{0,i} + a_{1,i} u_i + a_{2,i} v_i + \ldots + a_{n+1,i}u_i(u_i^2 + v_i^2)^{n-1} + \sum_{j\neq i}c_{j,i}v_j,\\
q_{i}(u,v)&= b_{0,i} + b_{1,i} u_i + b_{2,i} v_i + \ldots + b_{n+1,i}v_i(u_i^2 + v_i^2)^{n-1} +  \sum_{j\neq i}d_{j,i}u_j.\\
\end{aligned}
\end{equation}
\begin{thm}\label{main_thm_HC}
For general coefficients $a_{k,i}, b_{k,i}, c_{j,i}, d_{j,i}$ the number of complex solutions $(u,v)\in \CC^{2N}$ of the system \eqref{N_oscillators} is $(4n-3)^N$. This also holds if $a_{k,i}, b_{k,i}, c_{j,i}, d_{j,i}$ are determined by a general choice of the physical parameters $\omega_0,\omega,\lambda,\gamma,\alpha_1,\ldots\alpha_k$ from (\ref{FG_original}) (where we have one choice for each $i$).
\end{thm}
\begin{proof}
First, let us consider system of the form \eqref{N_oscillators} with all parameters $c_{j,i}, d_{j,i}$ specialized to $0$ and $a_{j,i}, b_{j,i}$ general. Such system corresponds to decoupled resonators and is given by $N$ independent blocks of type \eqref{equations_one_oscillator}. By Theorem \ref{thm:count1}, the number of its complex solutions is $(4n-3)^N$. Since with the specialization of the coefficients the number of solutions can only decrease (see, e.g., \cite[Theorem 7.1.4]{Sommese:Wampler:2005}),  the number of isolated solutions of general system \eqref{N_oscillators} is at least $(4n-3)^N$.

For the upper bound, it is proved in Lemma \ref{Section5.2} below that 
\begin{align*}
\left\{ s_i,\ s_iu_i,\ s_iv_i,\ s_iu_i(u_i^2 + v_i^2), \ldots,\  s_iv_i(u_i^2 + v_i^2)^{n-1}\right\}
\end{align*} 
is a Khovanskii basis of the subalgebra they generate for each $i=1,...,N$. Then by Theorem \ref{thm: Decoupling}, the polynomials
\begin{align*}
\{ s,\ &su_1,\ sv_1,\ su_1(u_1^2 + v_1^2), \ldots,\  sv_1(u_1^2 + v_1^2)^{n-1},\\
&su_2,\ sv_2,\ su_2(u_2^2 + v_2^2), \ldots,\  sv_2(u_2^2 + v_1^2)^{n-1},\\
&\vdots \\
&su_N,\ sv_N,\ su_N(u_N^2 + v_N^2), \ldots,\  sv_N(u_N^2 + v_N^2)^{n-1}\}
\end{align*} 
also form a Khovanskii basis of the subalgebra they generate. Applying Corollary~\ref{bound_for_same_blocks}, we obtain that the number of solutions is at most $(4n-3)^N$.
\end{proof}

For the proof of Theorem \ref{thm:count1}, we first prove that $4n-3$ is a lower bound and second that $4n-3$ is also an upper bound for the number of solutions.

\subsection{Lower bound for Theorem \ref{thm:count1}}\label{proof_lower_bound}
Denote $r^2 := u^2 + v^2$. 
In the two equations of Theorem \ref{thm:count1}, we set $a_1 = \omega_0^2-\omega^2, a_2 = -b_1 = a_{n+1} = b_{n+1} = 1$ and the other parameters to zero. We get:
$$
(\omega_0^2-\omega^2)u + v + ur^{2n-2}=
-u + vr^{2n-2}= 0.
$$ 
Notice that $(0,0)$ is a solution. Moreover, if $u = 0$, we also have $v = 0$. Conversely, if $v = 0$, we also have $u = 0$. So let us assume $u, v \neq 0$. We plug $u=vr^{2n-2}$ into the first equation to get 
\begin{align*}
v\left(1 + (\omega_0^2-\omega^2)r^{2n-2} + r^{4n-4}\right) = 0.
\end{align*}
Treating $r^2$ as a variable this has $2n-2$ distinct solutions for $\omega_0^2-\omega^2 \neq \pm 2$. Denote as~$s:=r^2$ such a solution. This gives the system
$
u^2 + v^2 -s= 
-u +vs^{n-1}= 0
$
with~$2$ solutions for $\omega_0^2-\omega^2 \neq 0$. In total, we therefore have $2 \cdot (2n-2) + 1 = 4n - 3$ isolated solutions in~$\CC^2$ for general $\omega_0^2-\omega^2$. Since substituting coefficients by complex numbers can only decrease the number of isolated solutions (see, e.g., \cite[Theorem 7.1.4]{Sommese:Wampler:2005}), this shows that $4n-3$ is a lower bound. Our choices of parameters satisfy the linear relations in (\ref{linear_relations}). Thus, the lower bound holds also in this case.

\subsection{Upper bound for Theorem \ref{thm:count1}}
When $a_0,b_0\neq 0$, and other parameters are general, then a zero $(u,v)$ must satisfy $u,v\neq 0$; i.e., it must be in the torus $(\mathbb C\setminus\{0\})\times (\mathbb C\setminus\{0\})$. We can therefore use Theorem \ref{updatedKush} to compute an upper bound for the number of zeros for general $a_0,\ldots,b_{n+1}$. 

The main task is thus to prove that
\begin{align*}
\left\{ s,\ su,\ sv,\ su(u^2 + v^2),\ sv(u^2 + v^2),\ \ldots,\ su(u^2 + v^2)^{n-1},\ sv(u^2 + v^2)^{n-1}\right\}
\end{align*} 
is a Khovanskii basis of the subalgebra that these polynomials generate. For this, we consider a monomial order $\succ$ with $u \succ v$. We write 
$$h_{0} = 1 \quad \text{and} \quad h_{2k-1}=u(u^2+v^2)^{k-1},\ h_{2k}=v(u^2+v^2)^{k-1},\ k\geq 1.$$ 
The leading terms of the polynomials~$sh_0, sh_1, sh_2, \ldots, sh_{2n}$ with respect to the chosen order $\succ$ are respectively
$$s,\ su,\ sv,\ su^3,\ svu^2,\ su^5,\ svu^4,\ \ldots,\ su^{2n-1},\ svu^{2n-2}.$$
The convex hull of the exponents of these leading monomials is the trapezoid.

\bigskip
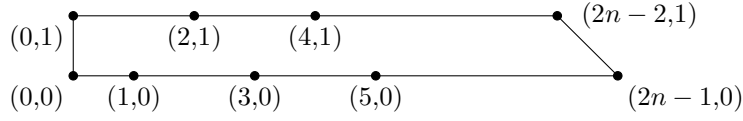
\begin{center}
\resizebox{0.8\columnwidth }{!}{
   \begin{tikzpicture}[scale=0.80]
	\draw[domain=0:9] plot (\x,{0}); 
	\draw[-] (0,0) -- (0,1);
	\draw[domain=0:8] plot (\x,{1}); 
	\draw[domain=8:9] plot (\x,{-\x + 9});
	
	\filldraw[black] (0,0) circle (2pt) node[anchor= north east] {(0,0)};
	\filldraw[black] (1,0) circle (2pt) node[anchor= north] {(1,0)};
	\filldraw[black] (3,0) circle (2pt) node[anchor= north] {(3,0)};
	\filldraw[black] (5,0) circle (2pt) node[anchor= north] {(5,0)};
	\filldraw[black] (2,1) circle (2pt) node[anchor= north] {(2,1)};
	\filldraw[black] (4,1) circle (2pt) node[anchor= north] {(4,1)};
	\filldraw[black] (0,1) circle (2pt) node[anchor= north east] {(0,1)};
	\filldraw[black] (9,0) circle (2pt) node[anchor= north west] {($2n-1$,0)};
	\filldraw[black] (8,1) circle (2pt) node[anchor= west] {\;\;($2n-2$,1)};
\end{tikzpicture}
}   
\end{center}
\bigskip

\noindent Its normalised volume is $(2n-1)\cdot 2 - 1 = 4n-3$. Therefore, if we can show that $sh_0,\ldots,sh_{2n}$ form a Khovanski basis, Theorem \ref{updatedKush} implies that $4n-3$ is an upper bound for the number of solutions.

Our proof strategy is based on the method by Kaveh and Manon that we have recalled in Theorem~\ref{thmKavehManon}. Let $z_0,z_1,\ldots,z_{2n}$ be coordinates on $\PP^{2n}$ and $Y \subset \PP^{2n}$ be a projective toric variety obtained from the monomial map
$$\phi\colon (s, u, v) \mapsto [\underset{z_0}{s}: \underset{z_1}{su}: \underset{z_2}{sv}: \underset{z_3}{su^3}: \underset{z_4}{svu^2}: \underset{z_5}{su^5}: \underset{z_6}{svu^4}: \ldots: \underset{z_{2n-1}}{su^{2n-1}}: \underset{z_{2n}}{svu^{2n-2}}] \in \PP^{2n}.$$
In the first step, we determine the generators of the toric ideal $I(Y)$ (Lemma~\ref{subs:idealgens}). Then, we show that after substituting the $sh_i$'s into the generating binomials of~$I(Y)$ they reduce to zero  (Lemma~\ref{subs:subdgens}), so that the $sh_i$'s indeed form a Khovanskii basis by Theorem \ref{thmKavehManon}. 

\begin{lem}\label{subs:idealgens}
Consider the polynomials
\begin{equation} \label{gens}
\begin{aligned}
& g_m := z_0^2z_m - z_1^2z_{m-2}\quad  \text{ for } 2n \geq m \geq 3, \\
& f_{l,m} := \begin{cases}
z_lz_m - z_{l+1}z_{m-1},& \text{ if } l+m \text{ is odd } \\
z_lz_m - z_{l+2}z_{m-2},& \text{ if } l+m \text{ is even } 
\end{cases}\quad  \text{ for } 2n \geq m \geq l+3 \geq 4. 
\end{aligned}
\end{equation}
They form a Gröbner basis of $I(Y)$ with respect to the DegLex monomial ordering ($z_0 > z_1 > \ldots > z_{2n}$). In particular,
$$I(Y) = \langle g_3,\ \ldots,\  g_{2n},\  f_{1,4},\  \ldots, f_{2n-3,2n}\rangle.$$
\end{lem}
\begin{proof}
One can check directly that $g_m\circ  \phi = 0$ and $f_{l,m}\circ\phi = 0$ for all $l,m$. This shows $g_m\in I(Y)$ and $f_{l,m}\in I(Y)$. Observe that 
$$\mathrm{LT}(g_m) = z_0^2z_m \quad\text{and}\quad \mathrm{LT}(f_{l,m}) =z_lz_m.$$
Let $p \in I(Y)$ be an arbitrary homogeneous binomial. Denote by $l$ the smallest index of the variables in $p$. 
We show that 
\begin{enumerate}
    \item if $l=0$, then $\mathrm{LT}(g_m) \mid \mathrm{LT}(p)$ for some $m\geq l+3$;
    \item if $l\geq 1$, then $\mathrm{LT}(f_{l,m}) \mid \mathrm{LT}(p)$ for some $m\geq l+3$.
\end{enumerate} 
By construction, $\mathrm{LT}(p)$ depends on~$z_l$. 
The idea of the proof is to show that~$\mathrm{LT}(p)$ depends also on~$z_m$ with~$m \geq l + 3$, and that $z_0$ appears with an even power. Consequently, if $l \geq 1$, we then have $\mathrm{LT}(f_{l,m}) \mid \mathrm{LT}(p)$, and if $l=0$, since the power of $z_0$ in the leading term of $p$ is even, we have $\mathrm{LT}(g_{k}) \mid \mathrm{LT}(p)$.

We consider the map $\Phi\colon \ZZ^{2n+1} \rightarrow \ZZ^3$ defined by
$$\Phi(c_0, \ldots, c_{2n}) = c_0 
\begin{bmatrix}1\\0\\ 0\end{bmatrix} + c_1 \begin{bmatrix}1\\1\\ 0\end{bmatrix} + c_2 \begin{bmatrix}1\\0\\ 1\end{bmatrix} +  \cdots + c_{2n-1}\begin{bmatrix}1\\2n-1\\ 0\end{bmatrix}  + c_{2n} \begin{bmatrix}1\\2n-2\\ 1\end{bmatrix},$$
where vectors to the right are the exponents of the monomials defining the monomial map above. Then, e.g. by \cite[Theorem 2.19]{telenToric}, 
$$I(Y) = \langle z^{m_{+}} - z^{m_{-}} \ | \ m \in \ker \Phi\rangle.$$
Notice that $m = (m_0, \ldots, m_{2n}) \in \ker \Phi$, if and only if $(m_0, \ldots, m_{2n})$ is a solution of the system of linear equations:
\begin{align*}
0&= m_0 + \ldots + m_{2n}, \\
0&= m_2 + m_4 + m_6 + \ldots + m_{2n},\\
0&= m_1 + 3m_3 + 2m_4 + 5m_5 + \ldots + (2n-2)m_{2n}.
\end{align*}
This is equivalent to
\begin{align*}
m_0 &= 2m_3 + 2m_4 + 4m_5 + 4m_6 + \ldots +(2n-2)m_{2n-1}+(2n-2)m_{2n}, \\
m_1 &= -3m_3 - 2m_4 - 5m_5 - \ldots - (2n-2)m_{2n}, \\
m_2 &= -m_4 - m_6 - \ldots - m_{2n}.
\end{align*}
In particular $m_0$ is even, so the power of $z_0$ appearing in the generators of $I(Y)$ is always even.

It remains to show that $\mathrm{LT}(p)$ depends on~$z_m$ with~$m \geq l + 3$. Suppose that the contrary holds and that the leading term of $p$ only depends on $z_l$, $z_{l+1}$ and $z_{l+2}$. Then
$$p = z_l^az_{l+1}^bz_{l+2}^c - \prod\limits_{q \geq l+1} z_q^{d_q} \text{ for } a > 0 \text{ and } a + b + c = \sum\limits_{q \geq l+1} d_q.$$
First, we immediately obtain a contradiction when $l=0$. Indeed, in this case, the monomial $(z_l^az_{l+1}^bz_{l+2}^c)\circ \phi$ is of degree strictly smaller than $a+b+c$ and cannot equal $(\prod_{q \geq l+1} z_q^{d_q})\circ\phi$, which is of degree greater or equal to $\sum_{q \geq l+1}d_q=a+b+c$. Thus, from now on we suppose $l>0$.

We can assume that the variables occur only once in the binomial. Therefore,~$z_{l+1}$ and $z_{l+2}$ can occur in the second term only if $b = 0$ or $c = 0$, respectively.

In fact, the variable $z_{l+1}$ never occurs in the second term. Indeed, if $b=0$ and~$l$ is odd, note that
$v \nmid (z_l\circ \phi)^a(z_{l+2}\circ \phi)^c$ and $v \mid (z_{l+1} \circ \phi).$ This implies $q\geq l+2$, because $p\circ\phi$ must be identically zero.
On the other hand, when $b=0$ and $l$ is even, $v^{a+c} \mid (z_l\circ \phi)^a(z_{l+2}\circ \phi)^c$, and since $a + c = \sum d_q$ and $v$ only appears in $z_q$ with $q$ even, every $q$ must be even. Again we get $q\geq l +2$.

Let us now denote by $A$ the power of $u$ in $(z_l\circ \phi)^a(z_{l+1}\circ \phi)^b(z_{l+2}\circ \phi)^c$ and by~$B$ the power of $u$ in the second term $\prod_{q \geq l+2} (z_q\circ \phi)^{d_q}$. As $p\circ\phi=0$ we must have 
$$A=B.$$
If $l$ is odd, then
$$A:=la + (l-1)b + (l+2)c\quad\text{and} \quad B:=\sum\limits_{q\geq l+2 \text{ odd}} q \cdot d_q + \sum\limits_{q\geq l+2 \text{ even}} (q-2) \cdot d_q.$$
Since $v$ must appear in equal quantities in both terms, we have 
$b = \sum_{q\geq l+2 \text{ even}} d_q$.
Furthermore, from homogeneity we obtain $a + c = \sum_{q \geq l+2 \text{ odd}} d_q.$ Since $q \geq l+2$,
$$B =q(a+c) + (q-2)b \geq (l+2)(a+c) + lb > A,$$
which contradicts $A=B$.

Similarly, when $l$ is even we have
$$A:=(l-2)a + (l+1)b + lc\quad\text{and}\quad B:=\sum\limits_{q\geq l+2 \text{ odd}} q \cdot d_q + \sum\limits_{q\geq l+2 \text{ even}} (q-2) \cdot d_q.$$
Since $v$ must appear in equal quantities in both terms,
$a+c = \sum_{q\geq l+2 \text{ even}} d_q$, 
and from homogeneity we obtain $b = \sum_{q \text{ odd}} d_q.$ This yields
$$B = qb + (q-2)(a+c) \geq (l+2)b + l(a+c) > A,$$
which again contradicts $A=B$. Thus, $z_m$ for some $m\geq l$ appears in the leading term of $p$ and the proof is complete.
\end{proof}

\begin{lem}\label{subs:subdgens}\label{Section5.2}
The polynomials $\{sh_0,\ sh_1,\ sh_2,\ \ldots,\ sh_{2n}\}$ are a Khovanskii basis for  the subalgebra 
$$S:= \CC\left[sh_0,\ sh_1,\ sh_2,\ \ldots,\ sh_{2n}\right] \subset \CC[s, u, v].$$
\end{lem}
\begin{proof}
In the previous lemma, we showed that $I(Y)$ is generated by the polynomials $\{ g_3,\ \ldots,\  g_{2n},\  f_{1,4},\  \ldots, f_{2n-3,2n}\}$ listed in \eqref{gens}. 
By Theorem \ref{thmKavehManon}, it suffices to check that for any~$p \in \{ g_3,\ \ldots,\  g_{2n},\  f_{1,4},\  \ldots, f_{2n-3,2n}\}$ the subduction algorithm (Algorithm \ref{alg}) reduces~$p(sh_0,\ldots, sh_{2n})$ to zero.
One can check directly that
$$f_{l,m}(sh_0,\ldots, sh_{2n}) = 0 \quad \text{ for all } 2n \geq m \geq l+3 \geq 4.$$
For the other polynomials, we have
$$g_m(sh_0,\ldots, sh_{2n}) = (sh_0)^2sh_m - (sh_1)^2(sh_{m-2}).$$
Here, we consider the cases $m$ even and odd separately. First, assume $m = 2k$ is even. In this case,
$sh_m = sv(u^2 + v^2)^{k-1}$ and $sh_{m-2} = sv(u^2 + v^2)^{k -2}$, so that
\begin{align*}
g_{2k}(sh_0,\ldots, sh_{2n}) 
&= s^3v(u^2 + v^2)^{k-1} - s^3u^2v(u^2 + v^2)^{k -2}\\
&= s^3v^3(u^2 + v^2)^{k -2}\\
&=(sv)^2\left( sv(u^2 + v^2)^{k-2}\right)\\
&=(sh_2)^2 \cdot sh_{2k-2},
\end{align*}
so that $g_{2k}(sh_0,\ldots, sh_{2n}) $ reduces to zero in one step of Algorithm \ref{alg}.

If $m = 2k + 1$, we have $sh_m = su(u^2 + v^2)^{k} $ and $sh_{m-2} = su(u^2 + v^2)^{k -1}.$
Hence,
\begin{align*}
g_m(sh_0,\ldots, sh_{2n}) 
&= s^3u(u^2 + v^2)^{k} - s^3u^3(u^2 + v^2)^{k -1}\\
&= s^3v^2u(u^2 + v^2)^{k -1}\\
&= (sv)^2\left( su(u^2 + v^2)^{k -1}\right)\\
&=(sh_2)^2 \cdot sh_{2k-1}.
\end{align*}
Also in this case Algorithm \ref{alg} takes one step to reduce $g_{2k}(sh_0,\ldots, sh_{2n}) $ to zero.
\end{proof}

\smallskip

\begin{rem}
In some situations researchers disregard the ``middle nonlinearities'' in the polynomials in \eqref{equations_one_oscillator} and identify some of the remaining coefficients obtaining the equations
\begin{align*}
au(u^2 + v^2)^{n-1} + bu + cv  &= 0, \\
av(u^2 + v^2)^{n-1} + dv - cu &= 0.
\end{align*}
Our proof shows that for general $a,b,c,d$ this system of equations also has $4n-3$ solutions over the complex numbers: the upper bound still works, since we have obtained the above system by specializing coefficients in \eqref{equations_one_oscillator}. It remains to note that the example used in proving the lower bound in Section \ref{proof_lower_bound} is a special case of the system above.

However, it is also possible to provide the basis of the corresponding toric ideal, subducting to zero. Let $Y \subset \PP^{4}$ be the projective toric variety obtained from the monomial map
$\phi\colon (s, u, v) \mapsto [{s}: {su}: {sv}: {su^{2n-1}}: {svu^{2n-2}}] \in \PP^{4}.$ Using coordinates $z_0,\ldots,z_4$ for $ \PP^{4}$ we have 
$$I(Y) = \langle z_1z_4 - z_2z_3,\ z_0^{2n-2}z_3 - z_1^{2n-1},\ z_0^{2n-2}z_4 - z_1^{2n-2}z_2\rangle.$$ For subduction, the first generator turns to zero after substitution of the tested polynomials. The other two can be checked by induction on $n$.

\end{rem}

\bigskip
\section{Multiple frequencies}\label{Section6}
In this section, we estimate the number of complex solutions for systems of polynomial equations that arise for $N=1$ resonator with $M>1$ frequencies. We obtain a lower bound for arbitrary $M$ and give an algorithm for computer verification of the upper bound for specific $M$.

Recall from (\ref{equations_M}) that for one resonator we have $2M$ equations in $2M$ variables $u = (u_1,...,u_M)$, $v= (v_1,...,v_M)$, which come in $M$ pairs:
$$
    \begin{aligned}
         & \begin{bmatrix}
p_1(u,v), &q_1(u,v),&\hdots ,&p_M(u,v), &q_M(u,v)    
\end{bmatrix}=0,\\
p_k(u,v)&= a_{k,0} + a_{k,1} u_k + c_k v_k + d_k u_k(u_k^2+v_k^2) + 2d_k \sum\limits_{j\neq k} u_k(u_j^2 + v_j^2),\\	
q_k(u,v)&= b_{k,0} - c_k u_k + b_{k,1} v_k + d_k v_k(u_k^2+v_k^2) + 2d_k \sum\limits_{j\neq k} v_k(u_j^2 + v_j^2).\\		
    \end{aligned}
$$
\begin{prop}
For general coefficients the system of equations above has at least $5^{MN}$ solutions. For $M=2$ and $N=1$ there are exactly $5^{MN}=25$ solutions.
\end{prop}
\begin{rem}
The proof for the upper bound in the case $M=2$ and $N=1$ is computational and can be, in principle, repeated for any other fixed $M\geq 2$ and $N\geq1$ by the same strategy that we present in Section \ref{sec:upper_bound_M2} below. 

Verifying that polynomials of \eqref{equations_M} are a Khovanskii basis for arbitrary $M$ is a challenging task. We leave it for future work.
\end{rem}
\subsection{Lower bound}
For each $k=1,...,M$ let us set the coefficients $a_{k,1},c_k,d_k=1$, the rest equal to zero and denote $r_k^2 = u_k^2 + v_k^2$. We obtain the system with~$M$ blocks:
\begin{equation*}
k = 1,...,M\colon \quad
    \begin{aligned}
u_k + v_k + u_k (r_k^2 + 2 \sum\limits_{j\neq k} r_j^2 ) = 0,\\	
- u_k + v_k (r_k^2 + 2 \sum\limits_{j\neq k} r_j^2 ) = 0.\\		\end{aligned}
\end{equation*}
Plugging $u_k = v_k(r_k^2 + 2 \sum\limits_{j\neq k} r_j^2 )$ into the first equation, we get 
$$v_k(r_k^2 + 2 \sum\limits_{j\neq k} r_j^2 )+ v_k + v_k (r_k^2 + 2 \sum\limits_{j\neq k} r_j^2 )^2=0,$$ which implies $v_k=0$ or $(r_k^2 + 2 \sum\limits_{j\neq k} r_j^2 )=\frac{-1 \pm i\sqrt{3}}{2}$. If $v_k=0$ we have $u_k=0$ and $r_k=0$. In other case, we treat $r_1^2,...,r_M^2$ as variables and solve linear systems 
\begin{equation}\label{linear_sys}
    (r_k^2 + 2 \sum\limits_{j\neq k} r_j^2 )=\frac{-1 \pm i\sqrt{3}}{2},
\end{equation}
each of which has a unique solution. With fixed $r_1^2 = s_1,...,r_M^2 = s_M$ for every pair $(u_k,v_k)\neq (0,0)$ we obtain a system $u_k^2 + v_k^2 - s_k=v_k + u_k(s_k + 2 \sum\limits_{j\neq k}s_j)=0$ with~$2$~solutions (note that $s_k + 2 \sum_{j\neq k}s_j \neq \pm i$). In total, \eqref{linear_sys} gives us 4 solutions.

As a result, together we have 
$$\sum_{k=0}^M{M \choose k}4^k \cdot 1^{M-k} = 5^M$$ 
complex solutions of the system \eqref{equations_M}. Similarly to Section \ref{proof_lower_bound}, $5^{MN}$ is the lower bound for $N$ coupled resonators with multifrequency ansatz.

\subsection{Upper bound in the case $M=2$ and $N=1$}\label{sec:upper_bound_M2}

Consider first $N=1$ resonator with $M=2$ leading frequencies. Then, the corresponding algebraic equations are
\begin{align*}
\left(\frac{\omega_{0}^{2}-\omega_1^{2}}{2}-\frac{\lambda\omega_{0}^{2}}{4}\right)u_1 + \frac{\gamma \omega_1}{2} v_1 + \frac{3}{8} \alpha u_1(u_1^2 + v_1^2) + \frac{3}{4} \alpha u_1(u_2^2+v_2^2)& = 0, \\
\left(\frac{\omega_{0}^{2}-\omega_1^{2}}{2}+\frac{\lambda\omega_{0}^{2}}{4}\right)v_1 - \frac{\gamma \omega_1}{2} u_1 +\frac{3}{8} \alpha v_1(u_1^2 + v_1^2) +\frac{3}{4}\alpha v_1(u_2^2+v_2^2)& = 0,\\
\frac{\left(\omega_{0}^{2}-\omega_2^{2}\right)}{2}u_2 + \frac{\gamma \omega_2}{2}v_2 +\frac{3}{8} \alpha u_2(u_2^2 + v_2^2) + \frac{3}{4} \alpha u_2(u_1^2+v_1^2)& = 0,\\
\frac{\left(\omega_{0}^{2}-\omega_2^{2}\right)}{2}v_2 - \frac{\gamma \omega_2}{2}u_2 +\frac{3}{8} \alpha v_2(u_2^2 + v_2^2) +\frac{3}{4} \alpha v_2(u_1^2+v_1^2)& = 0.
\end{align*}
We apply the software \texttt{SubalgebraBases} in \texttt{Macaulay2} \cite{M2} to verify that 
\begin{align*}
    s_1,\ &s_1u_1,\  s_1v_1,\  s_1u_1(u_1^2 + v_1^2),\  s_1u_1(u_2^2+v_2^2),\ \\
    s_2,\  &s_2u_1,\  s_2v_1,\  s_2v_1(u_1^2 + v_1^2),\  s_2v_1(u_2^2+v_2^2),\ \\
    s_3,\  &s_3u_2,\  s_3v_2,\  s_3u_2(u_2^2 + v_2^2),\  s_3u_2(u_1^2+v_1^2),\  \\
    s_4,\  &s_4u_2,\  s_4v_2,\ s_4v_2(u_2^2 + v_2^2),\  s_4v_2(u_1^2+v_1^2)\\
\end{align*}
form a Khovanskii basis for the subalgebra they generate. By Theorem \ref{updatedBKK}, the number of complex solutions is bounded then by the mixed volumes of the corresponding polytopes. We use the programming language \texttt{Julia} \cite{Julia} and the package \texttt{MixedSubdivisions} \cite{MixedSubdivisions} to compute this mixed volume:
\begin{lstlisting}[language=Julia]
using MixedSubdivisions
A = [[0 1 0 3 1; 0 0 1 0 0; 0 0 0 0 2; 0 0 0 0 0],
     [0 1 0 2 0; 0 0 1 1 1; 0 0 0 0 2; 0 0 0 0 0],
     [0 0 0 0 2; 0 0 0 0 0; 0 1 0 3 1; 0 0 1 0 0],
     [0 0 0 0 2; 0 0 0 0 0; 0 1 0 2 0; 0 0 1 1 1]]
mixed_volume(A)
\end{lstlisting}
The answer is $25 = 5^{2}$. 
\hfill $\diamond$ 

\begin{rem}
   We note that the approach of Section \ref{Section5} in this case is not enough. Again using the package \texttt{SubalgebraBases} in \texttt{Macaulay2}, one can verify that the polynomials
\begin{align*}
    \{s, &su_1, sv_1,  su_1(u_1^2 + v_1^2), su_1(u_2^2+v_2^2), sv_1(u_1^2 + v_1^2), sv_1(u_2^2+v_2^2)\\
    &su_2, sv_2, su_2(u_1^2 + v_1^2), su_2(u_2^2+v_2^2), sv_2(u_1^2 + v_1^2), sv_2(u_2^2+v_2^2)\}
\end{align*}
form a Khovanskii basis for the subalgebra they generate. But the normalized volume of the corresponding polytope is 33, which overestimates the number of solutions.
\end{rem}

\section{Summary and outlook}
    We have proven a variety of results that provides a tighter bound than the BKK bound for specific families of polynomial systems. Such polynomials often appear when solving for the stationary motion of a plethora of differential equations ubiquitously appearing in natural sciences. The transition from differential equations to algebraic polynomial equations relies on the Harmonic Balance method \cite{Krack_2019, Luo_2012}. By obtaining the tighter bound in this work, we pave the way towards more efficient Homotopy Continuation schemes \cite{hb_julia_code} that will allow for the solution of additional harmonics and higher polynomial potentials.


\bibliography{bibML}
\bibliographystyle{plain}
\end{document}